\documentclass[a4paper, reqno, 10pt]{amsart}
\usepackage{amssymb}
\usepackage{extarrows}
\usepackage{bm}
\usepackage[centering]{geometry}
\usepackage{enumitem}
\usepackage{longtable}
\usepackage{multirow}
\usepackage{array}
\usepackage{etex}
\usepackage{pictex}
\usepackage{graphics}

\usepackage[all]{xy}

\newcommand{\tabincell}[2]{\begin{tabular}{@{}#1@{}}#2\end{tabular}}
\numberwithin{equation}{section}
\usepackage{color}

\numberwithin{equation}{section}
\newtheorem{thm}{Theorem}[section]
\newtheorem{cor}[thm]{Corollary}
\newtheorem{lem}[thm]{Lemma}
\newtheorem{fact}[thm]{Fact}
\newtheorem{prop}[thm]{Proposition}
\theoremstyle{definition}

\newtheorem{rem}[thm]{Remark}

\def\Fib{\operatorname{Fib}}
\def\CoFib{\operatorname{CoFib}}
\def\Weq{\operatorname{Weq}}
\def\Hom{{\rm Hom}}
\def\Ext{{\rm Ext}}
\def\Ker{{\rm Ker}}
\def\Coker{{\rm Coker}}

\def\Fib{\operatorname{Fib}}
\def\CoFib{\operatorname{CoFib}}
\def\Weq{\operatorname{Weq}}

\begin{document}

\title[Model structures by modules of finite Gorenstein dimension]{Chains of model structures arising from \\ modules of finite Gorenstein dimension}

\author[Nan Gao, Xue-Song Lu, Pu Zhang]{Nan Gao$^*$, Xue-Song Lu, Pu Zhang$^*$}

\thanks{$^*$ Corresponding author}

\thanks{nangao@shu.edu.cn \ \ \ \ leocedar@sjtu.edu.cn \ \ \ \ pzhang$\symbol{64}$sjtu.edu.cn}

\thanks{Supported by National Natural Science Foundation of China, Grant No. 12131015, 12271333, and Natural Science Foundation of Shanghai under Grant No. 23ZR1435100.}


\subjclass[2010]{}

\begin{abstract}\ For any integer $n\ge 0$ and any ring $R$, \ $(\mathcal {PGF}_n, \ \mathcal P_n^\perp \cap \mathcal {PGF}^{\perp})$ proves to be a complete hereditary cotorsion pair in $R$-Mod, where
$\mathcal {PGF}$ is the class of PGF modules, introduced  by J. \v{S}aroch and J. \v{S}\'{t}ov\'{i}\v{c}ek,
and \ $\mathcal {PGF}_n$ is the class of $R$-modules of PGF dimension $\le n$.
For any Artin algebra $R$, \ $(\mathcal {GP}_n, \ \mathcal P_n^\perp \cap \mathcal {GP}^{\perp})$ proves to be a complete and hereditary cotorsion pair in $R$-Mod, where
$\mathcal {GP}_n$ is the class of modules of Gorenstein projective dimension $\le n$.
These cotorsion pairs induce two chains of hereditary Hovey triples \ $(\mathcal {PGF}_n, \ \mathcal P_n^\perp, \ \mathcal {PGF}^{\perp})$
and \ $(\mathcal {GP}_n, \ \mathcal P_n^\perp, \ \mathcal {GP}^{\perp})$,
and the corresponding homotopy categories in the same chain are the same.
It is observed that some complete cotorsion pairs in $R$-Mod can induce complete cotorsion pairs in some special extension closed subcategories of $R$-Mod.
Then corresponding results in exact categories $\mathcal {PGF}_n$, \ $\mathcal {GP}_n$, \ $\mathcal {GF}_n$, \
$\mathcal {PGF}^{<\infty}$, \ $\mathcal {GP}^{<\infty}$ and $\mathcal {GF}^{<\infty}$, are also obtained. As a byproduct,
$\mathcal{PGF} = \mathcal {GP}$ for a ring $R$ if and only if $\mathcal{PGF}^\perp\cap\mathcal{GP}_n=\mathcal P_n$ for some $n$.

\vskip5pt


\vskip5pt


\end{abstract}

\maketitle

\centerline{\bf Introduction}

\vskip 10pt

Gorenstein projective modules, introduced by
M. Auslander and M. Bridger \cite{ABr}, and E. Enochs and O. M. G. Jenda
\cite{EJ1}, have got deep interests and wide applications in mathematics.
A recent important progress is the introduction of projectively coresolved Gorenstein flat
(PGF, for short) modules over an arbitrary ring $R$, by J. \v{S}aroch and J. \v{S}\'{t}ov\'{i}\v{c}ek \cite{SS2}.

\vskip 10pt

This kind of PGF modules enjoys pleasant properties. It  is not clear whether a Gorenstein projective module is Gorenstein flat;
but PGF modules provide a common refinement of  Gorenstein projective modules and Gorenstein flat modules:
by definition they are  Gorenstein flat, and they also prove to be  Gorenstein projective (\cite[Theorem 4.4]{SS2}).
The full subcategory $\mathcal {PGF}$ of $R$-Mod consisting of the PGF modules  is a Frobenius category with
$\mathcal P$, the full subcategory of projective $R$-modules, as the class of projective-injective objects, and the stable triangulated category
$\mathcal {PGF}/\mathcal P$ is the homotopy category of infinite different abelian or exact model structures, in the sense of
D. Quillen \cite{Q1}, M. Hovey \cite{Hov}, and J. Gillespie \cite{G}.

\vskip5pt

What interesting and important are,  all the three classes $\mathcal {PGF}$, $\mathcal {GP}$, and $\mathcal {GF}$, are the left parts
of cotorsion pairs in $R$-Mod, where $\mathcal {GP}$ (respectively, $\mathcal {GF}$) is the full subcategory  of $R$-Mod consisting of
the Gorenstein projective (respectively, Gorenstein flat) modules.
A. Beligiannis and I. Reiten \cite[X, Theorem 2.4]{BR} prove that
$(\mathcal {GP}, \ \mathcal {GP}^{\perp})$ is a complete and hereditary
cotorsion pair in $R$-Mod, if $R$ is an Artin algebra.
\v{S}aroch and \v{S}\'{t}ov\'{i}\v{c}ek \cite[Theorem 4.9, Corollary 4.12]{SS2} prove that
$(\mathcal {PGF}, \ \mathcal {PGF}^\perp)$ and $(\mathcal {GF}, \ \mathcal {EC}\cap\mathcal {PGF}^\perp)$ are complete and hereditary cotorsion pairs in $R$-Mod,
for any ring $R$.
By a result of A. Iacob \cite[Proposition 9]{I} (see also \cite[Proposition 4.11]{M}), if $R$ is a Gorenstein ring (i.e., $R$ is  a left and right noetherian ring and
the injective dimensions of $_RR$ and $R_R$ are finite), then $\mathcal {PGF} = \mathcal {GP}$.
Thus, in this case, the cotorsion pair $(\mathcal {PGF}, \ \mathcal {PGF}^\perp)$
reads as $(\mathcal {GP}, \ \mathcal {P}^{<\infty})$, which has been given by Hovey \cite[Theorem 8.3]{Hov}.
By  the Hovey correspondence, all these three kinds of
cotorsion pairs induce abelian model structures in $R$-Mod.

\vskip5pt

Various homological dimensions provide more possibilities to obtain cotorsion pairs. For examples, for each non-negative integer $n$,
both $$(\mathcal P_n, \ \mathcal P_n^\perp) \ \ \ \ \mbox{ and} \ \ \ \ (\mathcal{F}_n, \ \mathcal{F}_n^{\perp})$$ are  complete and hereditary cotorsion pairs in $R$-Mod
(see \cite[Theorem 7.4.6]{EJ}, and \cite[Theorem 3.4(2)]{MD}, respectively),
where $\mathcal P_n$ (respectively, $\mathcal F_n)$ is the class of $R$-modules of projective (respectively, flat) dimension $\le n$.

\vskip5pt

The first aim of this paper is to show that these phenomena are quite common. For each non-negative integer $n$,
we prove that
$$(\mathcal {PGF}_n, \ \mathcal P_n^\perp\cap\mathcal {PGF}^\perp)$$
is a complete and hereditary cotorsion pair in $R$-Mod for any ring $R$, where $\mathcal {PGF}_n$ is the class of $R$-modules of PGF dimension $\le n$ (see Corollary \ref{mainthm2}); and that
$$(\mathcal {GP}_n, \ \mathcal P_n^\perp\cap \mathcal {P}^{<\infty})$$ is a complete and hereditary cotorsion pair in $R$-Mod, if $R$ is an Artin algebra, where $\mathcal {GP}_n$ is the class of $R$-modules of Gorenstein projective dimension $\le n$ (see Theorem \ref{mainthm4}).
R. El Maaouy \cite[Theorem A]{M} has been proved that there is a complete and hereditary cotorsion pair in $R$-Mod
$$(\mathcal{GF}_n, \  \mathcal{F}_n^{\perp} \cap \mathcal {PGF}^{\perp})$$
for any ring $R$, where $\mathcal {GF}_n$ is the class of $R$-modules of Gorenstein flat dimension $\le n$.
These three chains of cotorsion pairs induce three chains of hereditary abelian model structures on $R$-Mod:
$$(\mathcal {PGF}_n, \ \mathcal P_n^\perp, \ \mathcal {PGF}^\perp), \ \ \ \ \ (\mathcal {GP}_n, \ \mathcal P_n^\perp, \ \mathcal {P}^{<\infty}), \ \ \ \ \
(\mathcal {GF}_n, \ \mathcal{F}_n^\perp, \ \mathcal {PGF}^\perp).$$
 which are not trivial (i.e., not every module is a trivial object in the
model structure), and not projective (i.e., not every module is a fibrant object) in general.
(We remind that, in this paper,  a Hovey triple is written in the order \ $(\mathcal C, \mathcal F, \mathcal W)$,  rather than \ $(\mathcal C, \mathcal W, \mathcal F)$, where $\mathcal W$ is the class of trivial objects.)
Since all the hereditary abelian model structures in each chain have the same class of
trivial objects, the homotopy categories in the same chain are triangle equivalent to each other. Thus
the homotopy categories are  the stable categories
$\mathcal {PGF}/\mathcal P$ and $\mathcal {GP}/\mathcal P$, respectively.

\vskip5pt

One may further ask whether $\mathcal {PGF}^{<\infty}$, $\mathcal {GP}^{<\infty}$, and $\mathcal {GF}^{<\infty}$, can be the left parts
of complete and hereditary cotorsion pairs in $R$-Mod,
where $\mathcal {PGF}^{<\infty}$ (respectively, $\mathcal {GP}^{<\infty}, \ \mathcal {GF}^{<\infty}$) is the class  of $R$-modules of
finite PGF (respectively, the Gorenstein projective, Gorenstein flat) dimension.
This leads to an investigation on various finitistic dimensions. See Proposition \ref{finitisticdims} and Corollary \ref{finitedimensionversion}.

\vskip5pt

An overall landscape of this question can be viewed in the following table.

\vskip10pt

\centerline{\bf Table 1: \ Complete and hereditary cotorsion pairs}
\centerline{\bf \hskip50pt and the induced Hovey triples in $R$-Mod}

$${\tiny\begin{tabular}{|c|c|c|c|} \hline
\phantom{\LARGE 0} & {\tabincell{c}{$n=0$\\[3pt]}}& {\tabincell{c}{$n$\\[3pt]}}& {\tabincell{c}{$<\infty$\\[3pt]}}
\\[8pt]\hline
\tabincell{c}{$\mathcal P$}
& \tabincell{c}{\\ $(\mathcal P, \ R\mbox{-}{\rm Mod})$}
& \tabincell{c}{\\ $(\mathcal P_n, \ \mathcal P_n^\perp)$ \\ \\ Enochs, Jenda \cite[Theorem 7.4.6]{EJ}}
& \tabincell{c}{\\ If ${\rm Fpd} <\infty$, then \\ \\
\ $(\mathcal{P}^{<\infty}, \  (\mathcal{P}^{<\infty})^{\perp})$}
\\[15pt] \hline
\tabincell{c}{$\mathcal {FL}$ \\ $= \mathcal F_0$}
& \tabincell{c}{\\ $(\mathcal {FL}, \ \mathcal{EC})$\\ \\ Enochs, Jenda  \cite[Proposition 7.4.3]{EJ}}
& \tabincell{c}{\\ $(\mathcal F_n, \ \mathcal F_n^\perp)$\\ \\ Mao, Ding  \cite[Theorem 3.4(2)]{MD}}
& \tabincell{c}{\\ If ${\rm Fpd} <\infty$, then \\ \\
\ $(\mathcal{F}^{<\infty}, \  (\mathcal{F}^{<\infty})^{\perp})$}
\\[15pt] \hline
\tabincell{c}{$\mathcal {PGF}$}
& \tabincell{c}{\\ $(\mathcal {PGF}, \ \mathcal {PGF}^\perp)$ \\ \\ $(\mathcal {PGF}, \ R\mbox{-}{\rm Mod}, \ \mathcal {PGF}^\perp)$ \\ \\ \v{S}aroch, \v{S}\'{t}ov\'{i}\v{c}ek  \cite[Theorem 4.9]{SS2}}
& \tabincell{c}{\\ $(\mathcal {PGF}_n, \ \mathcal P_n^\perp\cap\mathcal {PGF}^\perp)$ \\ \\ $(\mathcal {PGF}_n, \ \mathcal P_n^\perp, \ \mathcal {PGF}^\perp)$ \\ \\ Corollary \ref{mainthm2}}
&\tabincell{c}{\\ If ${\rm Fpd} <\infty$, then \\ \\
\ $(\mathcal{PGF}^{<\infty}, \  (\mathcal{P}^{<\infty})^{\perp} \cap \mathcal {PGF}^{\perp})$
\\ \\ $(\mathcal{PGF}^{<\infty}, \  (\mathcal{P}^{<\infty})^{\perp}, \ \mathcal {PGF}^{\perp})$\\ \\ Corollary \ref{finitedimensionversion}}
\\[15pt] \hline
\tabincell{c}{$\mathcal {GP}$}
& \tabincell{c}{\\ For Artin algebras \\ \\ $(\mathcal {GP}, \ \mathcal {GP}^\perp)$ \\ \\ $(\mathcal {GP}, \ R\mbox{-}{\rm Mod}, \ \mathcal {GP}^\perp)$ \\ \\ Beligiannis, Reiten \\ \cite[X, Theorem 2.4(iv)]{BR}}
& \tabincell{c}{\\ For Artin algebras \\ \\ $(\mathcal {GP}_n, \ \mathcal P_n^\perp\cap\mathcal {GP}^{\perp})$
\\ \\ $(\mathcal {GP}_n, \ \mathcal P_n^\perp, \ \mathcal {GP}^{\perp})$
\\ \\ Theorem \ref{mainthm4}}
&\tabincell{c}{\\ For Artin algebras \\ \\ If ${\rm Fpd} <\infty$, then \\ \\ $(\mathcal {GP}^{<\infty}, \ (\mathcal P^{<\infty})^\perp\cap\mathcal {GP}^{\perp})$
\\ \\ $(\mathcal {GP}^{<\infty}, \ (\mathcal P^{<\infty})^\perp, \ \mathcal {GP}^{\perp})$ \\ \\ Corollary \ref{finitedimensionversion}}
\\[15pt] \hline
\tabincell{c}{$\mathcal{GF}$}
& \tabincell{c}{\\ $(\mathcal {GF}, \ \mathcal{EC}\cap\mathcal {PGF}^\perp)$
\\ \\ $(\mathcal {GF}, \ \mathcal{EC}, \ \mathcal {PGF}^\perp)$
\\ \\ \v{S}aroch, \v{S}\'{t}ov\'{i}\v{c}ek \\ \cite[Theorem 4.11, Corollary 4.12]{SS2}}
&\tabincell{c}{\\ $(\mathcal {GF}_n, \ \mathcal{F}_n^\perp\cap\mathcal {PGF}^\perp)$ \\ \\ $(\mathcal {GF}_n, \ \mathcal{F}_n^\perp, \ \mathcal {PGF}^\perp)$
\\ \\ Maaouy  \cite[Theorem A]{M}}
& \tabincell{c}{\\ If ${\rm Fpd} <\infty$, then \\ \\ $(\mathcal {GF}^{<\infty}, \ (\mathcal{F}^{<\infty})^\perp\cap\mathcal {PGF}^\perp)$
\\ \\ $(\mathcal {GF}^{<\infty}, \ (\mathcal{F}^{<\infty})^\perp, \ \mathcal {PGF}^\perp)$
\\ \\ Corollary \ref{finitedimensionversion}}
\\[15pt] \hline
\end{tabular}}$$

\vskip10pt

In order to study homology in non abelian categories, Quillen \cite{Q3} introduces exact categories. Any full subcategory of  an abelian category
which are closed under extensions and direct summands is a weakly idempotent complete exact category. By Gillespie \cite{G},
there is also the Hovey correspondence between the Hovey triples and the exact model structures, on any weakly idempotent complete exact category (see also \cite{S}).

\vskip5pt

The second aim of this paper is to look for complete and hereditary cotorsion pairs and exact model structures
in weakly idempotent complete exact categories
$\mathcal{PGF}_{n}$,  $\mathcal{GP}_{n}$,  $\mathcal{GF}_{n}$, and
$\mathcal{PGF}^{<\infty}$, $\mathcal{GP}^{<\infty}$ and $\mathcal{GF}^{<\infty}$, for any non-negative integer $n$.
This mainly comes from the following observation (see Theorem \ref{inducedctp1}):

\vskip5pt

If $\mathcal B$ is a full subcategory of abelian category $\mathcal A$, which is closed under extensions and the kernels of epimorphisms,
then any complete cotorsion pair  $(\mathcal X, \ \mathcal Y)$ in $\mathcal A$ with $\mathcal X\subseteq \mathcal B$ induces
complete cotorsion pair $(\mathcal X, \ \mathcal Y\cap\mathcal B)$ in exact category $\mathcal B.$ Moreover, if
$(\mathcal X, \ \mathcal Y)$ is hereditary in $\mathcal A$, then so is $(\mathcal X, \ \mathcal Y\cap\mathcal B)$ in $\mathcal B$.

\vskip5pt

This observation provides new complete cotorsion pairs and exact model structures in exact categories.
For example, for each non-negative integer $m$, one has two chains of complete and hereditary cotorsion pairs $(\mathcal X, \ \mathcal Y)$ in $R$-Mod with
$\mathcal X\subseteq \mathcal{PGF}^{<\infty}:$
$$(\mathcal P_m, \ \mathcal P_m^\perp), \ \ \ \ (\mathcal {PGF}_m, \ \mathcal P_m^\perp\cap \mathcal {PGF}^\perp).$$
Applying Theorem \ref{inducedctp1} one gets two chains of
complete and hereditary cotorsion pairs in $\mathcal{PGF}^{<\infty}$:
$$(\mathcal{P}_m, \ \mathcal P_m^\perp\cap \mathcal{PGF}^{<\infty}), \ \ \ \ \ (\mathcal{PGF}_m, \ \mathcal P_m^\perp\cap \mathcal{P}^{<\infty});$$
and  one chain of hereditary Hovey triples $(\mathcal{PGF}_m, \ \mathcal P_m^\perp\cap\mathcal{PGF}^{<\infty}, \ \mathcal P^{<\infty})$ in $\mathcal{PGF}^{<\infty}$.
See Corollary \ref{mainthm5.5}. When $m = 0$, this is \cite[Theorem 4.1] {DE} by G. Dalezios and I. Emmanouil.

\vskip5pt

Let us see the case in $\mathcal{GP}_n$. For any non-negative integer $n$,
one has  two chains of complete and hereditary cotorsion pairs $(\mathcal X, \ \mathcal Y)$ in $R$-Mod with
$\mathcal X\subseteq \mathcal{GP}_n$:
$$(\mathcal P_m, \ \mathcal P_m^\perp),  \ \ \ \ (\mathcal {PGF}_m, \ \mathcal P_m^\perp\cap \mathcal {PGF}^\perp),
$$
where $0\le m\le n$. Applying Theorem \ref{inducedctp1} one gets two chains of complete and hereditary cotorsion pairs in $\mathcal{GP}_n:$
$$(\mathcal{P}_m, \ \mathcal P_m^\perp\cap \mathcal{GP}_n), \ \ \ \
(\mathcal{PGF}_m, \ \mathcal P_m^\perp\cap \mathcal {PGF}^\perp\cap \mathcal{GP}_n);$$
and  one chain of hereditary Hovey triples $(\mathcal{PGF}_m, \ \mathcal P_m^\perp\cap\mathcal{GP}_n, \ \mathcal {PGF}^\perp\cap \mathcal{GP}_n)$  in $\mathcal{GP}_n$. See Corollary \ref{mainthm6}. In this way one gets double chains of complete and hereditary cotorsion pairs and hereditary Hovey triples, parameterized by $m$ and $n$.
All these exact model structures are not trivial, and not projective in general, with the same homotopy category $\mathcal {PGF}/\mathcal P$.

\vskip5pt

More interesting,  $(\mathcal{GP}_m, \ \mathcal P_m^\perp\cap\mathcal{P}_n)$ is a complete and hereditary cotorsion pair in $\mathcal{G P}_n$;
and $(\mathcal{G P}_m, \ \mathcal P_m^\perp\cap\mathcal{G P}_n, \ \mathcal P_{n})$ is a hereditary Hovey triple in $\mathcal{G P}_n$, for any integers $m$ and $n$ with $0\le m\leq n$.
See Theorem \ref{gpctp}. What  surprising is, Theorem \ref{gpctp} is for an arbitrary ring $R$, not only for Artin algebras.  Thus, Theorem \ref{gpctp}
is not an application of Theorem \ref{inducedctp1} and Theorem \ref{mainthm4} (which is for Artin algebras), and its proof is different from the one of Corollary \ref{mainthm6}.

\vskip5pt

As a byproduct, for a ring $R$, $\mathcal{PGF} = \mathcal {GP}$  if and only if $\mathcal{PGF}^\perp\cap\mathcal{GP}_n=\mathcal P_n$ for some non-negative integer $n$;
if and only if $\mathcal{PGF}^\perp\cap\mathcal{GP}_n=\mathcal P_n$ for any non-negative integer $n$.

\vskip15pt

\vskip5pt

An overall landscape can be viewed in the following table.

\vskip10pt

\centerline{\bf Table 2: \ Complete and hereditary cotorsion pairs and }
\centerline{\bf \hskip46pt the induced Hovey triples in exact categories}
$${\tiny\begin{tabular}{|c|c|c|c|} \hline
\phantom{\LARGE 0} & {\tabincell{c}{$n=0$\\[3pt]}}& {\tabincell{c}{$n$\\ $0\le m\le n$\\[3pt]}}& {\tabincell{c}{$<\infty$\\ $m$ \\[3pt]}}
\\[8pt]\hline
\tabincell{c}{\\ $\mathcal P$}
& \tabincell{c}{In $\mathcal P$ \\ $(\mathcal P, \ \mathcal P)$}
& \tabincell{c}{In $\mathcal P_n$ \\ $(\mathcal P_m, \ \mathcal P^\perp_m\cap\mathcal P_n)$}
& \tabincell{c}{In $\mathcal P^{<\infty}$ \\ $(\mathcal P_m, \ \mathcal P^\perp_m\cap\mathcal P^{<\infty})$}
\\[15pt] \hline
\tabincell{c}{ $\mathcal {FL} = \mathcal F_0$}
& \tabincell{c}{In $\mathcal {FL}$ \\  $(\mathcal {FL}, \ \mathcal{EC}\cap  \mathcal {FL})$}
& \tabincell{c}{ In $\mathcal F_n$ \\ $(\mathcal {P}_m, \ \mathcal{P}_m^\perp\cap \mathcal F_n)$ \\  \\$(\mathcal {F}_m, \ \mathcal{F}_m^\perp\cap \mathcal F_n)$}
& \tabincell{c}{ In $\mathcal F^{<\infty}$ \\$(\mathcal {P}_m, \ \mathcal{P}_m^\perp\cap \mathcal F^{<\infty})$\\ \\ $(\mathcal {F}_m, \ \mathcal{F}_m^\perp\cap \mathcal F^{<\infty})$}
\\[15pt] \hline
\tabincell{c}{$\mathcal {PGF}$}
& \tabincell{c}{In $\mathcal {PGF}$ \\ \\ $(\mathcal {PGF}, \ \mathcal {P})$ \\ \\ $(\mathcal {PGF}, \ \mathcal {PGF}, \ \mathcal {P})$ \\   \\  a special case of Corollary \ref{mainthm5}}
& \tabincell{c}{In $\mathcal {PGF}_n$ \\ \\ $(\mathcal {P}_m, \ \mathcal P_m^\perp\cap\mathcal {PGF}_n)$ \\ \\ $(\mathcal {PGF}_m, \ \mathcal P_m^\perp\cap\mathcal {P}_n)$
\\ \\  $(\mathcal {PGF}_m, \ \mathcal P_m^\perp\cap\mathcal {PGF}_n, \ \mathcal {P}_n)$
\\ \\ Corollary \ref{mainthm5}}
&\tabincell{c}{In $\mathcal {PGF}^{<\infty}$  \\ \\ $(\mathcal {P}_m, \ \mathcal P_m^\perp\cap\mathcal {PGF}^{<\infty})$ \\ \\ $ (\mathcal {PGF}_m, \ \mathcal P_m^\perp\cap \mathcal {P}^{<\infty})$
\\ \\ $ (\mathcal {PGF}_m, \ \mathcal P_m^\perp\cap\mathcal {PGF}^{<\infty}, \ \mathcal {P}^{<\infty})$
\\ \\ Corollary \ref{mainthm5.5}}
\\[15pt] \hline
\tabincell{c}{$\mathcal {GP}$}
& \tabincell{c}{In $\mathcal {GP}$ \\ \\ $(\mathcal {PGF}, \ \mathcal {PGF}^\perp\cap \mathcal {GP})$
\\ \\$(\mathcal {GP}, \ \mathcal {P})$
\\ \\ $(\mathcal {PGF}, \ \mathcal {GP}, \ \mathcal {PGF}^\perp\cap \mathcal {GP})$\\ \\ $(\mathcal {GP}, \ \mathcal {GP}, \ \mathcal {P})$ \\  \\ a special case of \\ Corollary \ref{mainthm6} and Theorem \ref{gpctp}}
& \tabincell{c}{In $\mathcal {GP}_n$ \\ \\ $(\mathcal {P}_m, \ \mathcal P_m^\perp\cap\mathcal {GP}_n)$
\\ \\  $(\mathcal {PGF}_m, \ \mathcal P_m^\perp\cap\mathcal {PGF}^\perp\cap \mathcal {GP}_n)$
\\ \\ $(\mathcal {GP}_m, \ \mathcal P_m^\perp\cap\mathcal {P}_n)$
\\ \\  $(\mathcal {PGF}_m, \ \mathcal P_m^\perp\cap \mathcal {GP}_n, \ \mathcal {PGF}^\perp\cap \mathcal {GP}_n)$
\\ \\ $(\mathcal {GP}_m, \ \mathcal P_m^\perp\cap\mathcal {GP}_n, \ \mathcal {P}_n)$
\\ \\ Corollary \ref{mainthm6} and Theorem \ref{gpctp}}
&\tabincell{c}{In $\mathcal {GP}^{<\infty}$ \\ \\  $(\mathcal {P}_m, \ \mathcal P_m^\perp\cap\mathcal {GP}^{<\infty})$
\\ \\$(\mathcal {PGF}_m, \ \mathcal P_m^\perp\cap\mathcal{PGF}^\perp\cap\mathcal {GP}^{<\infty})$
\\ \\ $(\mathcal {GP}_m, \ \mathcal P_m^\perp\cap\mathcal {P}^{<\infty})$
\\ \\ $(\mathcal{PGF}_m, \ \mathcal P_m^\perp\cap\mathcal{G P}^{<\infty}, \ \mathcal{PGF}^\perp\cap \mathcal{GP}^{<\infty})$
\\ \\ $(\mathcal {GP}_m, \ \mathcal P_m^\perp\cap\mathcal {GP}^{<\infty}, \ \mathcal {P}^{<\infty})$
\\ \\ Corollary \ref{mainthm7}}
\\[15pt] \hline
\tabincell{c}{$\mathcal{GF}$}
& \tabincell{c}{In $\mathcal {GF}$ \\ $(\mathcal {FL}, \ \mathcal {EC}\cap\mathcal{GF})$ \\ $(\mathcal {GF}, \ \mathcal{FL}\cap\mathcal {EC})$
\\ $(\mathcal {GF}, \ \mathcal {EC}\cap\mathcal {GF}, \ \mathcal{FL})$
\\ a special case of Corollary \ref{mainthm8}
\\ \\ $(\mathcal {PGF}, \ \mathcal {FL})$
\\ $(\mathcal {PGF}, \ \mathcal {GF}, \ \mathcal {FL})$
\\ Dalezios, Emmanouil \\ \cite[Theorem 4.3]{DE}}
&\tabincell{c}{In $\mathcal {GF}_n$ \\ \\ $(\mathcal {P}_m, \ \mathcal P_m^\perp\cap\mathcal {GF}_n)$
\\ \\ $(\mathcal {F}_m, \ \mathcal F_m^\perp\cap\mathcal {GF}_n)$
\\ \\ $(\mathcal {PGF}_m, \ \mathcal{F}_n\cap \mathcal P_m^\perp)$
\\ \\ $(\mathcal {GF}_m, \ \mathcal{F}_n\cap \mathcal F_m^\perp)$
\\ \\ $(\mathcal {PGF}_m, \ \mathcal P_m^\perp\cap\mathcal {GF}_n, \ \mathcal{F}_n)$
\\ \\ $(\mathcal {GF}_m, \ \mathcal F_m^\perp\cap\mathcal {GF}_n, \ \mathcal{F}_n)$
\\ \\ Corollary \ref{mainthm8}}
& \tabincell{c}{In $\mathcal {GF}^{<\infty}$ \\ \\ $(\mathcal {P}_m, \ \mathcal P_m^\perp\cap\mathcal {GF}^{<\infty})$
\\ \\ $(\mathcal {F}_m, \ \mathcal F_m^\perp\cap\mathcal {GF}^{<\infty})$
\\ \\ $(\mathcal {PGF}_m, \ \mathcal{F}^{<\infty}\cap \mathcal P_m^\perp)$
\\ \\ $(\mathcal {GF}_m, \ \mathcal{F}^{<\infty}\cap \mathcal F_m^\perp)$
\\ \\ $(\mathcal {PGF}_m, \ \mathcal P_m^\perp\cap\mathcal {GF}^{<\infty}, \ \mathcal{F}^{<\infty})$
\\ \\ $(\mathcal {GF}_m, \ \mathcal F_m^\perp\cap\mathcal {GF}^{<\infty}, \ \mathcal{F}^{<\infty})$
\\ \\ Corollary \ref{mainthm9}}
\\[15pt] \hline
\end{tabular}}$$

\vskip10pt

The paper is organized as follows.

\vskip5pt

1. \ Preliminaries

2. \ Abelian model structures arising from $\mathcal {PGF}_n$

3. \ Abelian model structures arising from $\mathcal {GP}_n$

4. \ Several finitistic dimensions

\hskip15pt 4.1 \ Equalities on finitistic dimensions

\hskip15pt 4.2 \ Abelian model structures arising from $\mathcal {PGF}^{<\infty}$, $\mathcal {GP}^{<\infty}$, and $\mathcal {GF}^{<\infty}$

5. \ Induced complete cotorsion pairs in exact categories with applications

\hskip15pt 5.1 \ Induced complete cotorsion pairs in exact categories

\hskip15pt 5.2 \ Application 1: Exact model structures on $\mathcal{PGF}_n$

\hskip15pt 5.3 \ Application 2: Exact model structures on $\mathcal{PGF}^{<\infty}$

6. \ Exact model structures on $\mathcal{GP}_n$ and $\mathcal{GP}^{<\infty}$

\hskip15pt 6.1 \  Application 3: Exact model structures on $\mathcal{GP}_n$

\hskip15pt 6.2 \  A new chain of cotorsion pairs and exact model structures on $\mathcal{G P}_n$

\hskip15pt 6.3 \  A sufficient and necessary condition for $\mathcal{PGF} = \mathcal {GP}$

\hskip15pt 6.4 \ Exact model structures on $\mathcal{GP}^{<\infty}$

7. \ Exact model structures on $\mathcal{GF}_n$ and  $\mathcal{GF}^{<\infty}$

\hskip15pt 7.1 \ Application 4: Exact model structures on $\mathcal{GF}_n$

\hskip15pt 7.2 \ Application 5: Exact model structures on $\mathcal{GF}^{<\infty}$

\section{\bf Preliminaries}

\subsection{Modules of finite Gorenstein dimension} \ Throughout $R$ is a ring with identity, $R$-Mod the category of left $R$-modules,
and $\mathcal{P}(R)$, or simply $\mathcal{P}$, the class of projective $R$-modules. Denote by $\mathcal P^{<\infty}$ (respectively, $\mathcal F^{<\infty}$)
the class of  modules of finite projective (respectively, flat) dimension. For a non-negative integer $n$, let $\mathcal{P}_{n}$ (respectively, $\mathcal{F}_{n}$) be the class of modules $M$
of projective dimension ${\rm pd}M\le n$ (respectively, flat dimension ${\rm fd}M\le n$).

\vskip 5pt

An $R$-module $M$ is {\it Gorenstein projective} (\cite{ABr}; \cite{EJ}) if it is
a  syzygy of an exact sequence of projective modules which remains exact after applying $\Hom_R(-, P)$ for any projective $R$-module $P$.
Let $\mathcal{GP}$ be the subcategory of Gorenstein projective $R$-modules.
{\it The Gorenstein projective dimension} of a module $M$ is denoted by
${\rm Gpd}M$. Let $\mathcal {GP}^{<\infty}$ be the class of modules of finite Gorenstein projective dimension.
For a non-negative integer $n$, let $\mathcal{GP}_{n}$ be the subcategory of $R$-modules $M$ with ${\rm Gpd}M\le n$.

\vskip 5pt

An $R$-module $M$ is {\it Gorenstein flat} (\cite{EJ}) if it is
a  syzygy of an exact sequence of flat modules which remains exact after tensoring with any right injective $R$-module.
It is open whether a Gorenstein projective module is Gorenstein flat.
Denote by $\mathcal{GF}$ the class of Gorenstein flat $R$-modules.
{\it The Gorenstein flat dimension} of module $M$ is denoted by
${\rm Gfd}M$ (\cite {Ho}). Let $\mathcal {GF}^{<\infty}$  be the class of modules of finite Gorenstein flat dimension.
For a non-negative integer $n$, let $\mathcal{GF}_{n}$ be the
class of modules $M$ of ${\rm Gfd}M\le n$.

\vskip 5pt

Following \cite{SS2}, an $R$-module $M$ is {\it projectively coresolved Gorenstein flat} (PGF for short),  if
it is a syzygy of an exact sequence of projective modules which remains exact after tensoring with any right injective $R$-module.
Denote by   $\mathcal{PGF}$  the subcategory of PGF modules. By
definition $\mathcal{PGF} \subseteq \mathcal{GF}.$ It is proved in \cite [Theorem 4.4]{SS2} that $\mathcal{PGF} \subseteq \mathcal{GP}.$
However, it is not clear whether a Gorenstein projective module is PGF. In fact, $\mathcal{PGF} = \mathcal{GP}$ if and only if $\mathcal{GP} \subseteq \mathcal{GF}$. Denoted by
${\rm PGFd}M$ {\it the {\rm PGF} dimension} of module $M$  (see \cite {DE}).
Denote by $\mathcal {PGF}^{<\infty}$ the class of modules of finite PGF dimension.
For a non-negative integer $n$, let $\mathcal{PGF}_{n}$ be the subcategory of $R$-modules $M$ with ${\rm PGFd}M\le n$.
Thus one has
$${\rm Gpd}M\le {\rm PGFd}M\le {\rm pd}M; \ \ \ \ \  {\rm fd}M\le {\rm pd}M; \ \ \ \ \ {\rm Gfd}M\le {\rm min}\{{\rm PGFd}M, \ {\rm fd}M\}$$
and
\[\xymatrix@C=30pt@R=0.4cm{& \mathcal F_n \ar@{^(->}[rd] &&&& \mathcal F^{<\infty}\ar@{^(->}[rd]
\\ \mathcal P_n \ar@{^(->}[ru] \ar@{^(->}[rd] && \mathcal {GF}_n && \mathcal P^{<\infty}\ar@{^(->}[ru] \ar@{^(->}[rd]&& \mathcal {GF}^{<\infty}
\\ & \mathcal {PGF}_n \ar@{^(->}[r]\ar@{^(->}[ru] & \mathcal {GP}_n &&& \mathcal {PGF}^{<\infty} \ar@{^(->}[r]\ar@{^(->}[ru] & \mathcal {GP}^{<\infty}.}
\]

\vskip 5pt

\subsection{Cotorsion pairs in exact categories} \ For {\it exact category} we refer to \cite{Q3} and \cite{K}. An exact category is {\it weakly idempotent complete}, if any splitting monomorphism has a cokernel; or equivalently,
any splitting epimorphism has a kernel. For more equivalent descriptions see \cite[7.2, 7.6]{Bu}.
Any full subcategory of an abelian category which is closed under extensions and direct summands is a weakly idempotent complete exact category, with the natural exact structure,
but it is not an abelian category, in general. The weakly idempotent complete exact categories considered in this paper are
$\mathcal{PGF}_{n}$, $\mathcal{GP}_{n}$ and $\mathcal{GF}_{n}$, for each non-negative integer $n$, and
$\mathcal{PGF}^{<\infty}$, $\mathcal{GP}^{<\infty}$ and $\mathcal{GF}^{<\infty}$. They are not abelian in general.

\vskip 5pt

For a class $\mathcal C$  of objects of exact category $\mathcal A$, let
$\mathcal C^{\perp}$ be the class of objects $X$ with ${\rm Ext}^{1}_{\mathcal A}(C, X)=0, \ \forall \ C\in \mathcal C.$ Similarly for  $^{\perp}\mathcal{C}$.
A pair $(\mathcal C, \mathcal F)$ of classes of objects of $\mathcal{A}$ is {\it a cotorsion pair}, if $\mathcal C^{\perp} = \mathcal F$ and $\mathcal C = \ ^{\perp}\mathcal F$.
For a cotorsion pair $(\mathcal C, \mathcal F)$, \ $\mathcal C\cap\mathcal F$ is called its {\it kernal}.

\vskip 5pt

A cotorsion pair $(\mathcal C, \mathcal F)$ is {\it complete}, if for any object $X\in \mathcal A$, there are admissible exact sequences
$$0\longrightarrow F\longrightarrow C\longrightarrow X\longrightarrow 0, \quad \text{and}\quad
0\longrightarrow X\longrightarrow F'\longrightarrow C'\longrightarrow 0,$$
with $C, \ C'\in \mathcal C$, and \ $F, \ F'\in \mathcal F$. For
an exact category with enough projective objects and enough injective objects, the two conditions in the completeness of a cotorsion pair
are equivalent. This is stated for $R$-Mod in \cite[Proposition 7.17]{EJ}; however the argument holds also for
exact categories with enough projective objects and injective objects. Recall that an object $P$ in exact category $\mathcal A = (\mathcal A, \mathcal E)$ is
$\mathcal E$-projective, or simply, projective, if for each admissible exact sequence
$0\longrightarrow X\longrightarrow Y\stackrel d\longrightarrow Z\longrightarrow 0$, the map $\Hom_\mathcal A(P, d)$ is surjective; and that $\mathcal A$ has enough projective objects, if
for any object $M\in\mathcal A$, there is an admissible exact sequence
$0\longrightarrow X\longrightarrow P\longrightarrow M\longrightarrow 0$ with $P$ projective.

\vskip 5pt

A cotorsion pair $(\mathcal C, \mathcal F)$ is {\it cogenerated by a set} if $\mathcal{Y}=\mathcal{S}^{\perp}$, where $\mathcal{S}$ is a set of modules (or, equivalently, cogenerated by a single module).
Any  cotorsion pair in $R$-Mod cogenerated by a set is complete (\cite [Theorem 10]{ET}; also \cite[Theorem 6.11(b)]{GT}: but in \cite{GT} ``cogenerated" is called ``generated"). This result has been generalized to
Grothendieck category with enough projective objects in \cite [Theorem 2.4]{Hov}.

\vskip 5pt

A cotorsion pair $(\mathcal C, \mathcal F)$ is {\it hereditary},
if $\mathcal C$ is closed under the kernel of deflations, and $\mathcal F$ is closed under the cokernel of inflations.
For a cotorsion pair $(\mathcal C, \mathcal F)$ in an abelian category  $\mathcal A$ with enough projective objects and injective objects,
these two conditions of the heredity are equivalent;
and they are equivalent to ${\Ext}^2_\mathcal A(\mathcal C, \mathcal F)=0$ for $C\in \mathcal C$ and $F\in \mathcal F;$ and also to
${\Ext}^i_\mathcal A(\mathcal C, \mathcal F)=0$ for $C\in \mathcal C$, $F\in \mathcal F$, and $i\geq 2$. See \cite [Proposition 1.2.10]{GR}.
This result holds also for a complete cotorsion pair in a weakly idempotent complete exact category. See \cite[Lemma 6.17]{S} (also \cite{SS}).

\subsection{Model structures} \ A closed model structure, or simply, {\it a model structure},  on a category $\mathcal M$, is a triple
\ $(\CoFib$, \ $\Fib$,  \ $\Weq)$ of classes of morphisms,
where the morphisms in the three classes are respectively called {\it cofibrations, fibrations}, and {\it weak equivalences},
satisfying Two out of three axiom, Retract axiom, Lifting axiom, and Factorization axiom.
The morphisms in ${\rm TCoFib}: = \CoFib\cap \Weq$ (respectively, ${\rm TFib}: =\Fib\cap \Weq$) are called {\it trivial cofibrations} (respectively, {\it trivial fibrations}).
For details we refer to \cite {Q1}, or \cite{Hov1}.
In a model structure any two classes of $\CoFib$, \  $\Fib$, \  $\Weq$ uniquely determine the third one (\cite{Q2}).

\vskip 5pt

For a model structure $(\CoFib$, \ $\Fib$,  \ $\Weq)$ on category $\mathcal M$ with zero object,
an object $X$ is {\it trivial} if $0 \longrightarrow X $ is a weak equivalence, or, equivalently,
$X\longrightarrow 0$ is a weak equivalence.
It is {\it cofibrant} if $0\longrightarrow X$ is a cofibration, and it is {\it fibrant} if $X\longrightarrow 0$ is a fibration.
An object is {\it trivially cofibrant} (respectively, {\it trivially fibrant}) if it is both trivial and  cofibrant (respectively, fibrant).
For a model structure  on category $\mathcal M$ with zero object, Quillen's {\it homotopy category} ${\rm Ho}(\mathcal M)$ is the localization $\mathcal M[\Weq^{-1}]$.
If $\mathcal M$ is an additive category, then ${\rm Ho}(\mathcal M)$ is a pretriangulated category in the sense of \cite{BR}. It is not necessarily a triangulated category, in general.
A model structure is {\it trivial} if each object is a trivial object.
The homotopy category of a model structure is zero if and only if it is a trivial model structure (\cite[Theorem 1.2.10(iv)]{Hov1}).
We are only interested in model structures which are not trivial.

\vskip 5pt

A model structure on an exact category is {\it exact} (\cite[3.1]{G}), if
cofibrations are exactly inflations with cofibrant cokernel, and fibrations are exactly
deflations with fibrant kernel. In this case, trivial cofibrations are exactly inflations with trivially cofibrant cokernel, and trivial fibrations are exactly
deflations with trivially fibrant kernel. If $\mathcal A$ is an abelian category, then an exact model structure on $\mathcal A$ is precisely {\it an abelian model structure} in \cite{Hov}.
An exact model structure is {\it projective} if each object is fibrant, or equivalently,
each trivially cofibrant object  is projective (\cite[4.5, 4.6]{G}).
\vskip5pt

{\it A Hovey triple} (\cite{Hov}, \cite{G}) in an exact category $\mathcal A$ is a triple \ $(\mathcal C, \mathcal F, \mathcal W)$ \ of classes of objects such that
\  $\mathcal W$ is {\it thick} in $\mathcal A$ (i.e., $\mathcal W$ is closed under direct summands,
and if two out of three terms in an admissible exact sequence  are in $\mathcal W$, then so is the third one); and that
both $(\mathcal C \cap \mathcal W, \ \mathcal F)$ and \ $(\mathcal C, \ \mathcal F \cap \mathcal W)$ \  are complete cotorsion pairs.
We stress that in this paper a Hovey triple is written in the order \ $(\mathcal C, \mathcal F, \mathcal W)$,  rather than \ $(\mathcal C, \mathcal W, \mathcal F)$.

\vskip 5pt

M. Hovey \cite{Hov} has established a one-one correspondence between abelian model structures and the Hovey triples in an abelian category. This has been extended to
weakly idempotent complete exact category in J. Gillespie \cite[Theorem 3.3]{G} (see also \cite[Theorem 6.9]{S}). Namely,
there is a one-to-one correspondence between exact model structures and the Hovey triples in weakly idempotent complete exact category $\mathcal A$, given by
$$({\rm CoFib}, \ {\rm Fib}, \ {\rm Weq})\mapsto (\mathcal{C}, \ \mathcal{F}, \ \mathcal W)$$
where \ $\mathcal C  = \{\mbox{cofibrant objects}\}, \ \
\mathcal F  = \{\mbox{fibrant objects}\}, \ \
\mathcal W  = \{\mbox {trivial objects}\}$, with the inverse \ \ $(\mathcal{C}, \ \mathcal{F}, \ \mathcal W) \mapsto ({\rm CoFib}, \ {\rm Fib}, \ {\rm Weq}),$ where
\begin{align*} &{\rm CoFib} = \{\mbox{inflation with cokernel in} \ \mathcal{C}\}, \ \ \
{\rm Fib}  = \{\mbox{deflation with kernel in}  \ \mathcal{F} \}, \\
& {\rm Weq}  = \{pi  \mid  i \ \mbox{is an inflation,} \ \Coker i\in \mathcal{C}\cap \mathcal W, \ p \ \mbox{is a deflation,} \ \Ker p\in \mathcal{F}\cap \mathcal W\}\end{align*}
and
$${\rm TCoFib} = \{\mbox{inflation with cokernel in} \ \mathcal C\cap \mathcal W\}, \ \ \ \  {\rm TFib} = \{\mbox{deflation with kernel in}  \ \mathcal{F}\cap \mathcal W \}.$$

\vskip5pt

Thus, we identify a Hovey triple with an exact model structure, in weakly idempotent complete exact category.
A Hovey triple  $(\mathcal C, \mathcal F, \mathcal W)$ is {\it hereditary} if
$(\mathcal C \cap \mathcal W, \ \mathcal F)$ and \ $(\mathcal C, \ \mathcal F \cap \mathcal W)$ \  are hereditary cotorsion pairs.
An advantage of a hereditary Hovey triple is that the homotopy category of the corresponding exact model structure
is just the stable category of a Frobenius category, and hence triangulated. See Theorem \ref{hoherhov} below.
Gillespie \cite[Theorem 1.1]{G2} has
constructed all the hereditary Hovey triples in an abelian category via two {\it compatible} hereditary complete cotorsion pairs.

\vskip5pt

Furthermore, H. Nakaoka and Y. Palu [NP] introduce an extriangulated category, which is a common generalization of an exact category and a triangulated category,
and generalize the Hovey correspondence to this setting up; moreover, remarkably,
the homotopy category is proved to be triangulated,
even if the two compatible complete cotorsion pairs are not necessarily hereditary. See \cite[Theorem 6.20]{NP}, and \cite[Theorem 6.34]{G3}.

\subsection{Homotopy categories of exact model structures}
The following observation is due to J. \v{S}\'{t}ov\'{i}\v{c}ek,
which  shows that the class of weak equivalences of an exact model structure is uniquely
determined only by the class of trivial objects, and hence its homotopy category
is uniquely determined only by $\mathcal W$. This is highly non-trivial.

\vskip5pt

A morphism $f$ in exact category is {\it admissible} if $f=ip$, where $p$ is a deflation and $i$ is an inflation.
In this case, $\Ker f$ and $\Coker f$ exist, and $\Ker f=\Ker p$ and $\Coker f=\Coker i$.
The following description of $\Weq$ in $(2)$ is similar to the multiplicative system determined by a Serre subcategory of an abelian category.

\begin{lem}\label{wdetermineweq} \  Let $\mathcal A$ be a weakly idempotent complete exact category, and $(\CoFib, \ \Fib, \ \Weq)$ an exact model structure on $\mathcal A$ given by Hovey triple \ $(\mathcal C, \ \mathcal F, \ \mathcal W)$. Then

\vskip5pt

$(1)$ \  {\rm(\cite[6.9(b)]{S})} \ $\Weq = \{d\circ i \ | \ i \ \mbox{is an inflation,} \ \Coker i\in \mathcal W, \ d \ \mbox{is a deflation,} \ \Ker d\in \mathcal W\}.$

\vskip5pt

$(2)$ \ Put \ $\mathcal S=\{f\in {\rm Mor}(\mathcal A) \ | \ f \ \text{is admissible},  \ \Ker f\in \mathcal W,  \ \Coker f\in \mathcal W\}.$
Then $$\Weq=\mathcal S\circ \mathcal S:=\{f\in {\rm Mor}(\mathcal A) \ | \ f=f_2f_1, \ f_i\in \mathcal S, \ i=1, 2\}.$$
In particular, $\mathcal S\subseteq \Weq$.
\end{lem}
\begin{proof} \ (2) \ Since this assertion seems to be not available in the references, we include a proof.
The inclusion $\Weq\subseteq \mathcal S\circ \mathcal S$ follow from (1); or alternatively, any weak equivalence $w$ has a decomposition
$w =pi$ where $i$ is a trivial cofibration and $p$ is a trivial fibration.
It remains to prove $\mathcal S\circ \mathcal S \subseteq \Weq$. It suffices to prove $\mathcal S\subseteq \Weq$.
Let $f$ be an admissible morphism with $\Ker f\in \mathcal W$ and $\Coker f\in \mathcal W$. Then $f=pi$, where $i$ is a trivial cofibration and $p$ is a fibration.
Thus $i$ is an inflation with $\Coker i\in \mathcal C\cap\mathcal W$ and $p$ is a deflation with $\Ker p\in \mathcal F$.
By \cite[Proposition 8.11]{Bu}, there is an admissible exact sequence
\[
0\longrightarrow \Ker f \longrightarrow \Ker p \longrightarrow
\Coker i \longrightarrow\Coker f\longrightarrow 0.
\]
Then one gets admissible exact sequences
\[
0\longrightarrow \Ker f \longrightarrow \Ker p \longrightarrow
M\longrightarrow 0 \ \ \ \ \ \text{and} \ \ \ 0\longrightarrow M  \longrightarrow
\Coker i \longrightarrow\Coker f\longrightarrow 0.\]
Since $\Coker i\in \mathcal W$ and $\Coker f\in \mathcal W$,  $M\in \mathcal W$. Since $M\in \mathcal W$ and $\Ker f\in \mathcal W$,
$\Ker p\in \mathcal W$. Thus $p$ is a deflation with $\Ker p\in \mathcal F\cap \mathcal W$, and hence $p$ is a trivial fibration. In particular, $p$ is a weak equivalence.
Therefore $f=pi\in \Weq$. \end{proof}

\begin{cor}\label{wdetermineho}  {\rm (\v{S}\'{t}ov\'{i}\v{c}ek)} \ Let $\mathcal A$ be a weakly idempotent complete exact category,
 $(\mathcal C,\mathcal F,\mathcal W)$ and $(\mathcal C',\mathcal F',\mathcal W')$  Hovey triples in $\mathcal A$. If $W=W'$,
then the homotopy categories of the two exact model structures on $\mathcal A$ are the same.
\end{cor}
\begin{proof} \ This follows from  Lemma \ref{wdetermineweq} and the definition of the homotopy category.
\end{proof}

\vskip5pt

For an exact model structure on weakly idempotent complete exact category $\mathcal A$, with cofiber sequences and fiber sequences in ${\rm Ho}(\mathcal A)$
(\cite[6.15]{S}; compare \cite[6.2]{Hov1}, or, cofibration sequences and fibration sequences in \cite{Q1}),
the homotopy category ${\rm Ho}(\mathcal A)$ is a pretriangulated category in the sense of \cite{BR}.
The following nice result can be explicitly founded in \v{S}\'{t}ov\'{i}\v{c}ek
\cite[Theorem 6.21]{S}. See also Gillespie \cite[Proposition 4.4, Proposition 5.2]{G} and Becker \cite[Proposition 1.1.14]{Bec}.

\begin{thm} \label{hoherhov} \ Let $\mathcal A$ be a weakly idempotent complete exact category, and  $(\mathcal C,  \ \mathcal F,  \ \mathcal W)$ a  hereditary Hovey triple in $\mathcal A$.
Then \ $\mathcal C \cap \mathcal F$ is a Frobenius category, with the induced exact structure$;$ \ and $\mathcal C \cap \mathcal F\cap \mathcal W$ is the class of projective-injective objects of
$\mathcal C \cap \mathcal F$.

\vskip5pt

Every cofiber sequence in ${\rm Ho}(\mathcal A)$ is isomorphic in ${\rm Ho}(\mathcal A)$ to a distinguished triangle in the stable category $(\mathcal C \cap \mathcal F)/(\mathcal C \cap \mathcal F\cap \mathcal W);$
and conversely every distinguished triangle in  the stable category is isomorphic in ${\rm Ho}(\mathcal A)$ to a cofiber sequence. Thus ${\rm Ho}(\mathcal A)$ is a triangulated category and
the composition $\mathcal C \cap \mathcal F \hookrightarrow \mathcal A \longrightarrow {\rm Ho}(\mathcal A)$ induces a triangle equivalence
$${\rm Ho}(\mathcal A)\cong (\mathcal C \cap \mathcal F)/(\mathcal C \cap \mathcal F\cap \mathcal W).$$
\vskip5pt

Moreover, every conflation $0\longrightarrow X \stackrel u\longrightarrow Y \stackrel v\longrightarrow Z \longrightarrow 0$ in $\mathcal A$ yields a cofiber sequence
$X \stackrel u\longrightarrow Y \stackrel v\longrightarrow Z\longrightarrow \Sigma X$ in ${\rm Ho}(\mathcal A);$ and each cofiber sequence in ${\rm Ho}(\mathcal A)$
is obtained in this way.
\end{thm}

\begin{cor} \label{wdeterminetriho} \ {\rm (\v{S}\'{t}ov\'{i}\v{c}ek)} \ Let $\mathcal A$ be a weakly idempotent complete exact category. Let $(\mathcal C,\mathcal F,\mathcal W)$ and $(\mathcal C',\mathcal F',\mathcal W')$ be two hereditary Hovey triples in $\mathcal A$. If $\mathcal W = \mathcal W'$, then $$(\mathcal C\cap \mathcal F)/(\mathcal C\cap \mathcal F\cap \mathcal W)\cong (\mathcal C'\cap \mathcal F')/(\mathcal C'\cap \mathcal F'\cap \mathcal W')$$ as triangulated categories.
\end{cor}
\begin{proof} \ By Corollary \ref{wdetermineho}, one only needs to show that the two triangulations in the same homotopy category also coincide. This is true, since by Theorem \ref{hoherhov}, the triangulation is
determined by the cofiber sequences, which are by definition uniquely determined by conflations in $\mathcal A$ and by $\mathcal W$. See \cite[Definition 6.15]{S}.
\end{proof}

\section{\bf Abelian model structures arising from $\mathcal {PGF}_n$}

The following remarkable result is due to \v{S}aroch and \v{S}\'{t}ov\'{i}\v{c}ek \cite{SS2}.

\begin{lem} \label{ctpPGF} \ {\rm (\cite[Theorem 4.9]{SS2})} \ Let $R$ be a ring. Then $\mathcal {PGF}^{\perp}$ is thick$;$ and
$(\mathcal{PGF}, \ \mathcal {PGF}^{\perp})$ is a complete and hereditary cotorsion pair with kernel $\mathcal P$, and cogenerated by a set. \end{lem}

\vskip5pt

The following interesting theorem can be found in E. E. Enochs and O. M. G. Jenda \cite{EJ}.

\begin{lem} \label{ctp7.4.6} \ {\rm (\cite[Theorem 7.4.6]{EJ})} \ Let $R$ be a ring, and $n$ a non-negative integer. Then
$(\mathcal{P}_n, \ \mathcal{P}_n^{\perp})$ is a complete and hereditary cotorsion pair, and cogenerated by a set.\end{lem}

\vskip10pt

The PGF version of Auslander - Buchweitz Theorem (\cite[Theorem 1.1]{AB}; see also \cite[Theorem 2.10]{Ho})
can be found in G. Dalezios and I. Emmanouil \cite{DE}.

\vskip5pt

\begin{lem} \label{PGFbasic} {\rm(\cite[Proposition 3.1]{DE})} \ Let $R$ be a ring, and $M$ an $R$-module with ${\rm PGFd} M = n$. Then
there is an exact sequence $0\longrightarrow K\longrightarrow G\stackrel \phi\longrightarrow M\longrightarrow 0$ with $G\in \mathcal {PGF}$ and
${\rm pd} K = n-1$. $($If $n = 0$ then $K = 0.)$   In particular, $\phi$ is a right $\mathcal {PGF}$-approximation of $M$.  \end{lem}

The following fact can be found in G. Dalezios, I. Emmanouil \cite{DE}.

\begin{lem} \label{PGFdim} \ {\rm (\cite[Proposition 2.4]{DE})} \ Let $R$ be a ring, and
$0\longrightarrow M_1 \longrightarrow M_2\longrightarrow M_3\longrightarrow 0$ an exact sequence. Then

\vskip5pt

${\rm (i)}$ \ \ ${\rm PGFd}M_1\le {\rm max}\{{\rm PGFd}M_2, {\rm PGFd}M_3\};$

 \vskip5pt

${\rm (ii)}$ \  ${\rm PGFd}M_2\le {\rm max}\{{\rm PGFd}M_1, {\rm PGFd}M_3\};$

\vskip5pt

${\rm (iii)}$ \  ${\rm PGFd}M_3\le {\rm max}\{{\rm PGFd}M_1+1, {\rm PGFd}M_2 \}$.
\end{lem}

The following result describes the modules of finite PGF dimension $\le n$. It is a theorem of \v{S}aroch - \v{S}\'{t}ov\'{i}\v{c}ek type:
\cite[Theorem 4.11]{SS2} is the similar result, for Gorenstein flat modules.

\begin{thm} \label{mainthm1} \ Let $R$ be a ring, $n$ a non-negative integer,  and $M$ an $R$-module. Then
$\mathcal {PGF}^{\perp}\cap\mathcal{PGF}_n  = \mathcal{P}_n;$ and the following are equivalent$:$

\vskip5pt

{\rm (1)} \ $M\in \mathcal{PGF}_n.$

\vskip5pt

{\rm (2)} \ There is an exact sequence
$0 \longrightarrow K \longrightarrow G \longrightarrow M \longrightarrow 0$ which is again exact after applying $\operatorname{Hom}_R(-, X)$ for $X\in \mathcal P_n^\perp$, where $G \in \mathcal{PGF}$ and $K \in \mathcal{P}_{n-1}$. $($If $n = 0$, then $K = 0.)$

\vskip5pt

{\rm (3)} \ $\operatorname{Ext}_R^1(M, C)=0$ for every module $C \in \mathcal{P}_n^{\perp}\cap \mathcal {PGF}^{\perp}.$

\vskip5pt

{\rm (4)} \ There is an exact sequence \ $0\longrightarrow M \longrightarrow L\longrightarrow  N\longrightarrow 0$ with $L \in \mathcal{P}_n$ and $N \in \mathcal{PGF}$.
\end{thm}

\begin{proof} \ $(1)\Longrightarrow (2)$:  \ Assume that $M\in \mathcal{PGF}_n.$ The case for $n=0$ is trivial. Assume that $n \geq 1$. By
Lemma \ref{PGFbasic} there is an exact sequence
$0\longrightarrow K \longrightarrow G \longrightarrow M \longrightarrow 0$
with $G \in \mathcal{PGF}$ and $K\in \mathcal{P}_{n-1}$.
It remains to  see that this exact sequence remains exact after applying $\operatorname{Hom}_R(-, X)$ for $X\in \mathcal P_n^\perp$.
Since $G \in \mathcal{PGF}$, by definition
there is an exact sequence $0\longrightarrow G \longrightarrow P \longrightarrow L\longrightarrow 0$ with $P$  projective and $L \in \mathcal{PGF}$.
Take the pushout

\[
\xymatrix@C=30pt{
  & & 0\ar[d]  & 0\ar[d] \\
0\ar[r] & K\ar[r]\ar@{=}[d]  & G\ar[r]\ar[d]  &  M\ar[r]\ar@{.>}[d] & 0\\
0\ar[r] & K \ar[r]  & P\ar@{.>}[r]\ar[d]  & H\ar[r]\ar[d] & 0\\
& &L \ar[d]\ar@{=}[r] & L\ar[d]\\
& & 0 &0 }
\]
It is clear that $H \in \mathcal{P}_n$. For any module $X\in \mathcal P_n^\perp$, in the following commutative diagram with exact rows

\[
\xymatrix@C=20pt{
0\ar[r] & \operatorname{Hom}_R(H, X)\ar[r]\ar[d]  &  \operatorname{Hom}_R(P, X)\ar[r]\ar[d]  &   \operatorname{Hom}_R(K, X) \ar[r]\ar@{=}[d] &  \operatorname{Ext}_{R}^{1}(H, X)\\
0\ar[r] & \operatorname{Hom}_R(M, X)\ar[r]  &  \operatorname{Hom}_R(G, X)\ar[r]  &   \operatorname{Hom}_R(K, X)}
\]
one has $\operatorname{Ext}_{R}^{1}(H, X) = 0$ since $H\in \mathcal{P}_n$ and $X\in \mathcal{P}_n^\perp$. It follows that the map
$\operatorname{Hom}_R(G, X)\longrightarrow \operatorname{Hom}_R(K, X)$ is surjective.

\vskip 5pt

$(2)\Longrightarrow (3):$  \ Assume that
$0 \longrightarrow K \stackrel f \longrightarrow G \longrightarrow M \longrightarrow 0$  is an exact sequence
which remains exact after applying $\operatorname{Hom}_R(-, X)$ for $X\in \mathcal P_n^\perp$, where $G \in \mathcal{PGF}$ and $K \in \mathcal{P}_{n-1}$. (If $n=0$ then $K=0$.)
For any module $C \in \mathcal{P}_n^{\perp}\cap \mathcal {PGF}^{\perp}$ one has $\Ext^1_R(G, C) = 0$.  Thus there is an exact sequence
$$\Hom_R(G, C)\stackrel {(f, C)} \longrightarrow \Hom_R(K, C) \longrightarrow {\rm Ext}_R^1(M, C)\longrightarrow 0.$$
But by the assumption the map  $(f, C)$ is surjective, it follows that ${\rm Ext}_R^1(M, C) = 0.$

\vskip 5pt

$(3)\Longrightarrow (4)$: \ Assume that $\operatorname{Ext}_R^1(M, C)=0$ for every module $C \in \mathcal{P}_n^{\perp}\cap \mathcal {PGF}^{\perp}.$
By Lemma~\ref{ctpPGF} one gets an exact sequence $0\longrightarrow M \longrightarrow L\longrightarrow N\longrightarrow 0$ with $L\in \mathcal {PGF}^{\perp}$ and $N\in \mathcal{PGF}$. Since $\operatorname{Ext}_R^1(N, C)=0$ for  $C \in \mathcal{P}_n^{\perp}\cap \mathcal {PGF}^{\perp},$
it follows that ${\rm Ext}^{1}_{R}(L, C)=0$ for  $C\in \mathcal{P}_n^{\perp}\cap \mathcal {PGF}^{\perp}$.

\vskip 5pt

By Lemma~\ref{ctp7.4.6} there is an exact sequence
$0\longrightarrow \Ker \pi \longrightarrow L^{\prime}\stackrel \pi\longrightarrow L\longrightarrow 0$ with
$L^{\prime}\in \mathcal P_n$ and ${\rm Ker}\pi\in \mathcal P_n^\perp$. Since $\mathcal{PGF}\subseteq \mathcal {GP}$, it follows that
$L^{\prime}\in \mathcal P_n \subseteq \mathcal {GP}^\perp\subseteq \mathcal {PGF}^\perp$.
Since both $L$ and $L^{\prime}$ are in $\mathcal {PGF}^\perp$, it follows from the thickness of $\mathcal {PGF}^\perp$ that
${\rm Ker}\pi\in \mathcal {PGF}^\perp$, and hence ${\rm Ker}\pi\in \mathcal{P}_n^{\perp}\cap \mathcal {PGF}^\perp$. Therefore
${\rm Ext}^{1}_{R}(L, {\rm Ker}\pi)=0$, and then $\pi$ splits and $L\in \mathcal{P}_n$.

\vskip 5pt

$(4) \Longrightarrow (1)$:  Assume that \ $0\longrightarrow M \longrightarrow L\longrightarrow  N\longrightarrow 0$ is an exact sequence with $L \in \mathcal{P}_n$ and $N \in \mathcal{PGF}$. Since $L \in \mathcal{P}_n\subseteq \mathcal{PGF}_n$ and $N \in \mathcal{PGF}\subseteq \mathcal{PGF}_n$, it follows $M\in \mathcal{PGF}_n$ (c.f. Lemma \ref{PGFdim}).

\vskip5pt

Finally, we show $\mathcal {PGF}^{\perp}\cap\mathcal{PGF}_n = \mathcal{P}_n$. Since $\mathcal{P}_n\subseteq \mathcal {GP}^{\perp} \subseteq\mathcal {PGF}^{\perp}$ and $\mathcal{P}_n\subseteq \mathcal{PGF}_n$,
thus $\mathcal{P}_n\subseteq \mathcal {PGF}^{\perp}\cap\mathcal{PGF}_n$.
On the other hand, if $M \in \mathcal {PGF}^{\perp}\cap\mathcal{PGF}_n$, then by assertion (4) there is
an exact sequence  $0 \longrightarrow M \longrightarrow L \longrightarrow N \longrightarrow 0$ with $L\in\mathcal P_n$ and $N \in \mathcal{PGF}$.
Thus ${\rm Ext}^1_R(N, M) = 0.$ Hence this exact sequence splits and $M \in \mathcal{P}_n$.
\end{proof}

\vskip5pt

When $n = 0$, the following result coincides with \cite[Theorem 4.9]{SS2}.

\begin{cor} \label{mainthm2} \ Let $R$ be a ring and $n$ a non-negative integer. Then

\vskip5pt

{\rm (1)} \ The pair \ $(\mathcal{PGF}_n, \  \mathcal{P}_n^{\perp} \cap \mathcal {PGF}^{\perp})$ is a complete and hereditary cotorsion pair in $R\mbox{-}{\rm Mod}$ with
kernel $\mathcal{P}_n\cap\mathcal{P}_n^{\perp}$. This cotorsion pair is cogenerated by a set.

\vskip 5pt

{\rm (2)}  \ The triple \ $(\mathcal{PGF}_n, \ \mathcal{P}_n^{\perp}, \ \mathcal {PGF}^{\perp})$ is a hereditary Hovey triple in $R\mbox{-}{\rm Mod};$
$\mathcal{PGF}_n\cap\mathcal{P}_n^{\perp}$ is a Frobenius category such that $\mathcal P_n\cap \mathcal P_n^{\perp}$ is the class of projective-injective objects$;$ and the homotopy category
is the stable category $(\mathcal{PGF}_n\cap \mathcal{P}_n^{\perp})/(\mathcal  P_n\cap \mathcal{P}_n^{\perp})$, which is triangle equivalent to \ $\mathcal{PGF}/\mathcal P.$

\vskip 5pt

{\rm (3)}  \ If $n\ge 1$, then the corresponding  abelian model structure is projective if and only if $\mathcal{P}^{<\infty} = \mathcal{P}$.
\end{cor}

\begin{proof} \ (1) \ Put \  $\mathcal{A}=\mathcal{PGF}_n$ and $\mathcal{B} =\mathcal{P}_n^{\perp} \cap \mathcal {PGF}^{\perp}$. By Theorem \ref{mainthm1}(3) one has $\mathcal{A}={ }^{\perp} \mathcal{B}$ and $\mathcal B\subseteq \mathcal A^\perp$. On the other hand, one has
$$\mathcal{P}_n\cup \mathcal{PGF}\subseteq \mathcal{PGF}_n\cup \mathcal{PGF} = \mathcal{PGF}_n.$$ Thus
$$\mathcal{A}^{\perp}= \mathcal{PGF}_n^{\perp} \subseteq (\mathcal P_n\cup \mathcal{PGF})^\perp = \mathcal{P}_n^{\perp} \cap \mathcal {PGF}^{\perp} = \mathcal{B}.$$
Hence $\mathcal{B} = \mathcal{A}^{\perp}$ and $(\mathcal{A}, \mathcal{B})$ is a cotorsion pair.
It is hereditary, since $\mathcal{PGF}_n$ is closed under taking kernels of epimorphisms (cf. Lemma \ref{PGFdim}); or equivalently
$\mathcal{P}_n^{\perp}$, and hence $\mathcal{P}_n^{\perp} \cap \mathcal {PGF}^{\perp}$,  is closed under taking cokernels of monomorphisms.

\vskip 5pt

To prove that it is complete, it suffices to show that it is cogenerated by a set.
By  \cite[Theorem 4.9]{SS2} (cf. Lemma \ref{ctpPGF}), the cotorsion pair \ $(\mathcal{PGF}, \ \mathcal {PGF}^{\perp})$ is cogenerated by a set, say  $T$.
Also the cotorsion pair \ $(\mathcal{P}_n, \mathcal{P}_n^{\perp})$ is cogenerated by a set, say  $B$. Then
$\mathcal{P}_n^{\perp}\cap \mathcal {PGF}^{\perp} = B^{\perp} \cap T^{\perp} = (B \cup T)^{\perp}$, i.e.,
the cotorsion pair \ $(\mathcal{PGF}_n, \  \mathcal{P}_n^{\perp} \cap \mathcal {PGF}^{\perp})$ is cogenerated by the set  $B \cup T$.

\vskip 5pt

By  Theorem \ref{mainthm1}, $\mathcal{PGF}_n\cap \mathcal {PGF}^{\perp} = \mathcal{P}_n$, and hence the kernel of this cotorsion pair is  $\mathcal{P}_n\cap\mathcal{P}_n^{\perp}$.

\vskip 5pt

(2) \ Put $(\mathcal C, \ \mathcal F, \ \mathcal W) = (\mathcal{PGF}_n, \ \mathcal{P}_n^{\perp}, \ \mathcal {PGF}^{\perp})$.
Then $\mathcal W = \mathcal {PGF}^{\perp}$ is thick (cf. Lemma \ref{ctpPGF}). By  Theorem \ref{mainthm1} and Lemma~\ref{ctp7.4.6},
 $(\mathcal C\cap \mathcal W, \ \mathcal F) = (\mathcal{P}_n, \ \mathcal{P}_n^{\perp})$ is a complete and hereditary
cotorsion pair;  and $(\mathcal C, \ \mathcal F\cap \mathcal W) = (\mathcal{PGF}_n, \ \mathcal{P}_n^{\perp}\cap\mathcal {PGF}^{\perp})$ is a complete and hereditary
cotorsion pair. By definition $(\mathcal{PGF}_n, \ \mathcal{P}_n^{\perp}, \ \mathcal {PGF}^{\perp})$ is a hereditary Hovey triple in $R\mbox{-}{\rm Mod}$.

\vskip 5pt

By Theorem \ref{hoherhov}, $\mathcal C\cap \mathcal F = \mathcal{PGF}_n\cap \mathcal{P}_n^{\perp}$
is a Frobenius category, and $\mathcal C\cap \mathcal F\cap \mathcal W = \mathcal P_n\cap \mathcal P_n^{\perp}$ is its class of projective-injective objects$;$ and the homotopy category
is the stable category $(\mathcal{PGF}_n\cap \mathcal{P}_n^{\perp})/(\mathcal  P_n\cap \mathcal{P}_n^{\perp})$. Since
$(\mathcal{PGF}, \ R\mbox{-}{\rm Mod}, \ \mathcal {PGF}^{\perp})$ is also a hereditary Hovey triple in $R\mbox{-}{\rm Mod}$, it follows  from  Corollary \ref{wdeterminetriho} that there is a triangle equivalence
$(\mathcal{PGF}_n\cap \mathcal{P}_n^{\perp})/(\mathcal  P_n\cap \mathcal{P}_n^{\perp})\cong \mathcal{PGF}/\mathcal P.$

\vskip 5pt

(3) \ The class of fibrant objects of this model structure is \ $\mathcal{P}_n^{\perp}$. By definition this model structure is projective if and only if
$\mathcal{P}_n^{\perp} = R\mbox{-}{\rm Mod}.$ This is equivalent to say
$\mathcal{P}_n = \ ^\perp(\mathcal{P}_n^{\perp}) = \ ^\perp (R\mbox{-}{\rm Mod}) = \mathcal P$, and this is equivalent to $\mathcal{P}^{<\infty} = \mathcal{P}$, if $n\ge 1$.
\end{proof}

\begin{rem} The completeness of the cotorsion pair in Corollary \ref{mainthm2} can also be proved by the definition. See the proof of Theorem \ref{mainthm4}. \end{rem}

\section{\bf Abelian model structures arising from $\mathcal {GP}_n$}

The following nice result is due to A. Beligiannis and I. Reiten \cite{BR}.

\begin{lem} \label{ctpGP} \ {\rm (\cite[X, Theorem 2.4(iv)]{BR})} \ Let $A$ be an Artin algebra. Then \ $(\mathcal{GP}, \ \mathcal {GP}^{\perp})$ is a complete and hereditary cotorsion pair in $A$\mbox{-}{\rm Mod}.
Moreover,  $\mathcal {GP}^{\perp}$ is thick in $A$\mbox{-}{\rm Mod} $($this does not need the assumption of Artin algebra$)$. \end{lem}
\begin{proof} \ We include a proof for the thickness, for it is not mentioned in \cite[X, Theorem 2.4(iv)]{BR}.
Let $0\longrightarrow X\longrightarrow Y\longrightarrow Z\longrightarrow 0$ be an exact sequence, and $G\in \mathcal{GP}$.
Then $G$ is a syzygy of a complete projective resolution $P^\bullet$. By exact sequence of complexes
\ $0\longrightarrow {\rm Hom}_A(P^\bullet, X)\longrightarrow {\rm Hom}_A(P^\bullet, Y)\longrightarrow {\rm Hom}_A(P^\bullet, Z)\longrightarrow 0,$
one sees that any two of ${\rm Hom}_A(P^\bullet, X), \ {\rm Hom}_A(P^\bullet, Y)$, \ ${\rm Hom}_A(P^\bullet, Z)$ are exact, so is the third one.
From this the thickness of $\mathcal {GP}^{\perp}$ follows. \end{proof}

\vskip5pt

Auslander - Buchweitz Theorem is fundamental in Gorenstein homological algebra.
See \cite[Theorem 1.1]{AB}, or more explicitly in \cite[Theorem 2.10]{Ho}.

\begin{lem} \label{basic} \ Let $R$ be a ring, and $M$ an $R$-module with ${\rm Gpd} M = n$. Then
there is an exact sequence $0\longrightarrow K\longrightarrow G\stackrel \phi\longrightarrow M\longrightarrow 0$ with $G\in \mathcal {GP}$ and
${\rm pd} K = n-1$. $($If $n = 0$ then $K = 0.)$   In particular, $\phi$ is a right $\mathcal {GP}$-approximation of $M$.
\end{lem}

\vskip5pt

The following is another theorem of \v{S}aroch - \v{S}\'{t}ov\'{i}\v{c}ek type (compare Theorem \ref{mainthm1} and \cite[Theorem 4.11]{SS2}).

\begin{thm} \label{mainthm3} \ Let $A$ be a ring, $n$ a non-negative integer,  and $M$ an $A$-module. Then \ $\mathcal {GP}^{\perp}  \cap\mathcal{GP}_n = \mathcal{P}_n;$ and one has the implications ${\rm (1)} \Longleftrightarrow {\rm (2)}\Longleftrightarrow {\rm (4)} \Longrightarrow {\rm (3)}$, where

\vskip5pt

{\rm (1)} \ $M\in \mathcal{GP}_n.$

\vskip5pt

{\rm (2)} \ There is an exact sequence
$0 \longrightarrow K \longrightarrow G \longrightarrow M \longrightarrow 0$
which is again exact after applying $\operatorname{Hom}_A(-, X)$ for $X\in \mathcal P_n^\perp$, where $G \in \mathcal{GP}$ and $K \in \mathcal{P}_{n-1}$. $($If $n = 0$, then $K = 0.)$

\vskip5pt

{\rm (3)} \ $\operatorname{Ext}_A^1(M, C)=0$ for every module $C \in \mathcal{P}_n^{\perp}\cap \mathcal {GP}^{\perp}.$

\vskip5pt

{\rm (4)} \ There is an exact sequence \ $0\longrightarrow M \longrightarrow L\longrightarrow  N\longrightarrow 0$ with $L \in \mathcal{P}_n$ and $N \in \mathcal{GP}$.

\vskip5pt

Moreover, if $A$ is an Artin algebra, then  ${\rm (3)} \Longrightarrow {\rm (4)}$,
and hence all the statements ${\rm (1)}, \ {\rm (2)}, \ {\rm (3)}, \ {\rm (4)}$ are equivalent.

\end{thm}

\begin{proof} \ It is clear that
$\mathcal{P}_n\subseteq \mathcal {GP}^{\perp}  \cap\mathcal{GP}_n$.
Conversely, if $M \in \mathcal {GP}^{\perp}  \cap\mathcal{GP}_n$, then by Lemma \ref{basic} there is an exact sequence $0\longrightarrow K\longrightarrow G \longrightarrow M\longrightarrow 0$ with $G\in \mathcal {GP}$ and
$K\in \mathcal P_{n-1}$.  Since $K\in \mathcal P_{n-1}\subseteq \mathcal {GP}^{\perp}$ and $M \in \mathcal {GP}^{\perp}$,
$G\in \mathcal {GP}^{\perp}$, and hence $G\in \mathcal {GP}\cap \mathcal {GP}^{\perp} = \mathcal P.$
Thus $M \in \mathcal{P}_n$.

\vskip5pt

$(1)\Longrightarrow (2)$:  \ Assume that $M\in \mathcal{GP}_n.$ The case for $n=0$ is trivial. Assume that $n \geq 1$. By Lemma \ref{basic} there is an exact sequence $0\longrightarrow K \longrightarrow G \longrightarrow M \longrightarrow 0$
with $G \in \mathcal{GP}$ and $K\in \mathcal{P}_{n-1}$. We show that this sequence remains exact after applying $\operatorname{Hom}_A(-, X)$ for $X\in \mathcal P_n^\perp$. Since $G$ is Gorenstein projective, by definition  there is
an exact sequence $0\longrightarrow G \longrightarrow P \longrightarrow G^{\prime} \longrightarrow 0$ with $P\in\mathcal P$  and $G^\prime \in \mathcal{GP}$.
Take the pushout

\[
\xymatrix@C=30pt@R=0.6cm{
  & & 0\ar[d]  & 0\ar[d] \\
0\ar[r] & K\ar[r]\ar@{=}[d]  & G\ar[r]\ar[d]  &  M\ar[r]\ar@{.>}[d] & 0\\
0\ar[r] & K \ar[r]  & P\ar@{.>}[r]\ar[d]  & H\ar[r]\ar[d] & 0\\
& &G^\prime \ar[d]\ar@{=}[r] & G^\prime\ar[d]\\
& & 0 &0 }
\]
Then $H \in \mathcal{P}_n$. For any module $X\in \mathcal P_n^\perp$, in the following commutative diagram with exact rows

\[
\xymatrix@C=23pt{
0\ar[r] & \operatorname{Hom}_A(H, X)\ar[r]\ar[d]  &  \operatorname{Hom}_A(P, X)\ar[r]\ar[d]  &   \operatorname{Hom}_A(K, X) \ar[r]\ar@{=}[d] &  \operatorname{Ext}_A^{1}(H, X)\\
0\ar[r] & \operatorname{Hom}_A(M, X)\ar[r]  &  \operatorname{Hom}_A(G, X)\ar[r]  &   \operatorname{Hom}_A(K, X)}
\]
one has $\operatorname{Ext}_A^{1}(H, X) = 0$. It follows that the map
$\operatorname{Hom}_A(G, X)\longrightarrow \operatorname{Hom}_A(K, X)$ is surjective.

\vskip 5pt

$(2)\Longrightarrow (4)$: \ The case when $n=0$ is trivial. If $n\geq 1$, by assumption there is an exact sequence
$0 \longrightarrow K \longrightarrow G \longrightarrow M \longrightarrow 0$ where $G \in \mathcal{GP}$ and $K \in \mathcal{P}_{n-1}$. Also there is an exact sequence $0 \longrightarrow G \longrightarrow P \longrightarrow N \longrightarrow 0$ where $P \in \mathcal{P}$ and $N \in \mathcal{GP}$. Take the pushout
\[
\xymatrix@C=30pt@R=0.6cm{
  & & 0\ar[d]  & 0\ar[d] \\
0\ar[r] & K\ar[r]\ar@{=}[d]  & G\ar[r]\ar[d]  &  M\ar[r]\ar@{.>}[d] & 0\\
0\ar[r] & K \ar[r]  & P\ar@{.>}[r]\ar[d]  & L\ar[r]\ar[d] & 0\\
& &N \ar[d]\ar@{=}[r] & N\ar[d]\\
& & 0 &0 }
\]
Since $K \in \mathcal{P}_{n-1}$ and $P\in \mathcal P$, $L \in \mathcal{P}_n$. We are done.

\vskip 5pt

$(4) \Longrightarrow (1)$:  Assume that \ $0\longrightarrow M \longrightarrow L\longrightarrow  N\longrightarrow 0$ is an exact sequence with $L \in \mathcal{P}_n$ and $N \in \mathcal{GP}$. Since $L \in \mathcal{P}_n\subseteq \mathcal{G P}_n$ and $N \in \mathcal{GP}\subseteq \mathcal{G P}_n$, it follows from  \cite[Theorem 2.20]{Ho} that $M\in \mathcal{G P}_n$.

\vskip 5pt

$(2)\Longrightarrow (3):$  \ Assume that
$0 \longrightarrow K \stackrel f \longrightarrow G \longrightarrow M \longrightarrow 0$ with $G \in \mathcal{GP}$ and $K \in \mathcal{P}_n$ is an exact sequence
which remains exact after applying  $\operatorname{Hom}_A(-, X)$ for $X\in \mathcal P_n^\perp$.
For any module $C \in \mathcal{P}_n^{\perp}\cap \mathcal {GP}^{\perp}$ one has $\Ext^1_A(G, C) = 0$.  Thus one has an exact sequence
$$\Hom_A(G, C)\stackrel {(f, C)} \longrightarrow \Hom_A(K, C) \longrightarrow {\rm Ext}_A^1(M, C)\longrightarrow 0.$$
But by the assumption the map  $(f, C)$ is surjective, it follows that ${\rm Ext}_A^1(M, C) = 0.$

\vskip 5pt

$(3)\Longrightarrow (4)$: \ For this, assume that $A$ is an Artin algebra. Assume that $\operatorname{Ext}_A^1(M, C)=0$ for every module $C \in \mathcal{P}_n^{\perp}\cap \mathcal {GP}^{\perp}.$
By Lemma~\ref{ctpGP} (This is the place we need the assumption that $A$ is an Artin algebra.),  one gets an exact sequence $0\longrightarrow M \longrightarrow L\longrightarrow N\longrightarrow 0$ with $L\in \mathcal {GP}^{\perp}$ and $N\in \mathcal{GP}$.
Since $\operatorname{Ext}_A^1(N, C)=0$ for  $C \in \mathcal{P}_n^{\perp}\cap \mathcal {GP}^{\perp},$
it follows that ${\rm Ext}^1_A(L, C)=0$ for  $C\in \mathcal{P}_n^{\perp}\cap \mathcal {GP}^{\perp}$.

\vskip 5pt

By Lemma~\ref{ctp7.4.6} one gets an exact sequence
$0\longrightarrow \Ker \pi \longrightarrow L^{\prime}\stackrel \pi\longrightarrow L\longrightarrow 0$ with
$L^{\prime}\in \mathcal P_n$ and ${\rm Ker}\pi\in \mathcal P_n^\perp$. Since both $L$ and $L^{\prime}$ are in $\mathcal{GP}^{\perp}$, and since $\mathcal{GP}^{\perp}$ (cf. Lemma \ref{ctpGP})
is thick in $A$-Mod, it follows that
${\rm Ker}\pi\in \mathcal {GP}^{\perp}$, and hence ${\rm Ker}\pi\in \mathcal{P}_n^{\perp}\cap \mathcal {GP}^{\perp}$. Therefore
${\rm Ext}^1_A(L, {\rm Ker}\pi)=0$, and then $\pi$ splits and $L\in \mathcal{P}_n$.
\end{proof}

Since we do not know whether  the cotorsion pair \ $(\mathcal{GP}, \ \mathcal {GP}^{\perp})$ is cogenerated by a set, the proof of the following theorem is in fact
not the same as Corollary \ref{mainthm2}.

\begin{thm} \label{mainthm4} \ Let $A$ be an Artin algebra and $n$ a non-negative integer. Then

\vskip5pt

{\rm (1)} \ The pair \ $(\mathcal{G P}_n, \  \mathcal{P}_n^{\perp} \cap \mathcal {GP}^{\perp})$ is a complete and hereditary cotorsion pair in $A\mbox{-}{\rm Mod}$ with kernel
$\mathcal{P}_n\cap\mathcal{P}_n^{\perp}$.

\vskip 5pt

{\rm (2)}  \ The triple \ $(\mathcal{G P}_n, \ \mathcal{P}_n^{\perp}, \ \mathcal {GP}^{\perp})$ is a hereditary Hovey triple in $A\mbox{-}{\rm Mod};$
$\mathcal{G P}_n\cap\mathcal{P}_n^{\perp}$ is a Frobenius category such that $\mathcal  P_n\cap \mathcal{P}_n^{\perp}$ is its class of projective-injective objects$;$ and the homotopy category
is the stable category $(\mathcal{GP}_n\cap \mathcal{P}_n^{\perp})/(\mathcal  P_n\cap \mathcal{P}_n^{\perp})$, which is triangle equivalent to \ $\mathcal{GP}/\mathcal P.$
\end{thm}

\begin{proof} \ (1) \ Put \  $\mathcal{A}=\mathcal{G P}_n$ and $\mathcal{B} =\mathcal{P}_n^{\perp} \cap \mathcal {GP}^{\perp}$. By Theorem \ref{mainthm3}(3) one has $\mathcal{A}={ }^{\perp} \mathcal{B}$ and $\mathcal B\subseteq \mathcal A^\perp$. On the other hand, one has \ $\mathcal{P}_n\cup \mathcal{GP}\subseteq \mathcal{GP}_n\cup \mathcal{GP} = \mathcal{G P}_n.$ Thus
$$\mathcal{A}^{\perp}= \mathcal{GP}_n^{\perp} \subseteq (\mathcal{P}_n\cup \mathcal{GP})^\perp = \mathcal{P}_n^{\perp} \cap \mathcal {GP}^{\perp} = \mathcal{B}.$$
Hence $\mathcal{B} = \mathcal{A}^{\perp}$ and $(\mathcal{A}, \mathcal{B})$ is a cotorsion pair.
It is hereditary, since $\mathcal{G P}_n$ is closed under taking kernels of epimorphisms (cf. \cite[Theorem 2.20]{Ho}); or equivalently
$\mathcal{P}_n^{\perp}$, and hence $\mathcal{P}_n^{\perp} \cap \mathcal {GP}^{\perp}$,  is closed under taking cokernels of monomorphisms.

\vskip 5pt

To prove that it is complete, it suffices to show that for $M\in A\mbox{-}{\rm Mod}$, there is an exact sequence
$$0 \longrightarrow L \longrightarrow G \longrightarrow M \longrightarrow 0$$ with $G\in \mathcal{G P}_n$ and $L\in \mathcal{P}_n^{\perp} \cap \mathcal {GP}^{\perp}$.
Since $(\mathcal{GP}, \ \mathcal {GP}^{\perp})$ is a complete cotorsion pair in $A$\mbox{-}{\rm Mod} (cf. Lemma \ref{ctpGP}),
there is an exact sequence
$$0 \longrightarrow N \longrightarrow H \longrightarrow M \longrightarrow 0$$ with $H\in \mathcal{GP}$ and $N\in \mathcal{GP}^{\perp}$.
Since $(\mathcal{P}_n, \  \mathcal{P}_n^{\perp})$ is a complete cotorsion pair in $A\mbox{-}{\rm Mod}$ (cf. Lemma \ref{ctp7.4.6}),
there is an exact sequence
$$0 \longrightarrow N \longrightarrow L \longrightarrow C \longrightarrow 0$$ with $L\in \mathcal{P}_n^{\perp}$ and $C\in  \mathcal {P}_n$.
Take the pushout

\[
\xymatrix@C=30pt@R=0.6cm{& 0\ar[d]  & 0\ar[d] \\
0\ar[r] & N\ar[r]\ar[d]  & H\ar[r]\ar@{.>}[d]  &  M\ar[r]\ar@{=}[d] & 0\\
0\ar[r] & L \ar@{.>}[r]\ar[d]  & G\ar[r]\ar[d]  & M\ar[r] & 0\\
& C \ar[d]\ar@{=}[r] & C\ar[d]\\
& 0 &0}
\]
Since $N\in \mathcal{GP}^{\perp}$ and $C\in \mathcal {P}_n\subseteq \mathcal{GP}^{\perp}$, it follows that $L\in \mathcal{GP}^{\perp}$, and hence
$L\in \mathcal{P}_n^{\perp} \cap \mathcal {GP}^{\perp}$. Since $H\in \mathcal{GP}\subseteq \mathcal{G P}_n$  and $C\in \mathcal {P}_n \subseteq \mathcal{G P}_n$, it follows that
$G\in \mathcal{G P}_n$. We are done.

\vskip 5pt

By Theorem~\ref{mainthm3}, \ $\mathcal{GP}_n  \cap\mathcal {GP}^{\perp} = \mathcal{P}_n,$  and hence the kernel of this cotorsion pair is  $\mathcal{P}_n\cap\mathcal{P}_n^{\perp}$.

\vskip 5pt

(2) \ Put $(\mathcal C, \ \mathcal F, \ \mathcal W) = (\mathcal{G P}_n, \ \mathcal{P}_n^{\perp}, \ \mathcal {GP}^{\perp})$.
Then $\mathcal W = \mathcal {GP}^{\perp}$ is thick, by Lemma~\ref{ctpGP}. Since $\mathcal{GP}_n  \cap\mathcal {GP}^{\perp} = \mathcal{P}_n,$
it follows from  Lemma~\ref{ctp7.4.6} that $(\mathcal C\cap \mathcal W, \ \mathcal F) = (\mathcal{P}_n, \ \mathcal{P}_n^{\perp})$ is a complete and hereditary
cotorsion pair in $A\mbox{-}{\rm Mod}$.  By (1), $(\mathcal C, \ \mathcal F\cap \mathcal W) = (\mathcal{GP}_n, \ \mathcal{P}_n^{\perp}\cap\mathcal {GP}^{\perp})$ is a complete and hereditary
cotorsion pair in $A\mbox{-}{\rm Mod}$. By definition $(\mathcal{G P}_n, \ \mathcal{P}_n^{\perp}, \ \mathcal {GP}^{\perp})$ is a hereditary Hovey triple in $A\mbox{-}{\rm Mod}$.
The remaining assertion follows from Theorem \ref{hoherhov} and Corollary \ref{wdeterminetriho}, using hereditary Hovey triple
$(\mathcal{G P}, \ R\mbox{-}{\rm Mod}, \ \mathcal {GP}^{\perp})$ in $R\mbox{-}{\rm Mod}$.
\end{proof}
\begin{rem} We do not need to consider the corresponding results for $\mathcal {GP}_n$ for Gorenstein ring $R$,
by using the cortorsion pair \ $(\mathcal {GP}, \ \mathcal P^{<\infty})$, since in this case one has $\mathcal {GP} = \mathcal {PGF}$ (cf. \cite[Proposition 9]{I}).
\end{rem}
\section{\bf Several finitistic dimensions}

\subsection{Equalities on finitistic dimensions} \ For any non-negative integer $n$ and any ring $R$, we have seen in Corollary \ref{mainthm2} that
$(\mathcal{PGF}_n, \  \mathcal{P}_n^{\perp} \cap \mathcal {PGF}^{\perp})$ is a complete and hereditary cotorsion pair in $R\mbox{-}{\rm Mod}$, and that
$(\mathcal{PGF}_n, \ \mathcal{P}_n^{\perp}, \ \mathcal {PGF}^{\perp})$ is a hereditary Hovey triple in $R\mbox{-}{\rm Mod}.$
Also, for Artin algebra $A$, we have seen in Theorem \ref{mainthm4} that
$(\mathcal{GP}_n, \  \mathcal{P}_n^{\perp} \cap \mathcal {GP}^{\perp})$ is a complete and hereditary cotorsion pair in $A\mbox{-}{\rm Mod}$, and that
$(\mathcal{GP}_n, \ \mathcal{P}_n^{\perp}, \ \mathcal {GP}^{\perp})$ is a hereditary Hovey triple in $A\mbox{-}{\rm Mod}.$

\vskip5pt

For any ring $R$,  Maaouy \cite[Theorem A]{M} proves that
$(\mathcal {GF}_n, \ \mathcal {F}_n^{\perp}\cap \mathcal{PGF}^{\perp})$ is a complete and hereditary cotorsion pair,
and that $(\mathcal {GF}_n, \ \mathcal {F}_n^{\perp}, \ \mathcal{PGF}^{\perp})$ is a hereditary Hovey triple, in $R$-Mod.

\vskip5pt

It seems to be natural to ask whether $\mathcal {PGF}^{<\infty}$, $\mathcal {GP}^{<\infty}$ and $\mathcal {GF}^{<\infty}$ could be the left parts of cotorsion pairs in $R$-Mod.
This problem reasonably leads to several finitistic dimensions.

\vskip5pt

The (left) finitistic  projective dimension ${\rm Fpd} = {\rm Fpd}(R)$ of a ring $R$ has been introduced by H. Bass  (\cite{Bass}), which is defined as
$${\rm Fpd} = {\rm sup}\{{\rm pd}M \ | \ M\in R\mbox{-}{\rm Mod}, \ {\rm pd}M <\infty\}.$$
Similarly, the (left) finitistic flat dimension ${\rm Ffd} = {\rm Ffd}(R)$ of $R$ is defined as
$${\rm Ffd} = {\rm sup}\{{\rm fd}M \ | \ M\in R\mbox{-}{\rm Mod}, \ {\rm fd}M <\infty\}.$$
The  (left) finitistic Gorenstein projective dimension ${\rm FGpd} = {\rm FGpd}(R)$ (\cite[2.27]{Ho}) is defined as
$${\rm FGpd} = {\rm sup}\{{\rm Gpd}M \ | \ M\in R\mbox{-}{\rm Mod}, \ {\rm Gpd}M <\infty\}.$$
The  (left) finitistic Gorenstein flat dimension FGfd = FGfd$(R)$  (\cite[3.24]{Ho}) is defined as
$${\rm FGfd} = {\rm sup}\{{\rm Gfd}M \ | \ M\in R\mbox{-}{\rm Mod}, \ {\rm Gfd}M <\infty\}.$$
Also, one can define the  (left) finitistic PGF dimension of $R$ as
$${\rm FPGFd} = {\rm sup}\{{\rm PGFd}M \ | \ M\in R\mbox{-}{\rm Mod}, \ {\rm PGFd}M <\infty\}.$$

\begin{prop} \label{finitisticdims} \ Let $R$ be a ring. Then

\vskip5pt

{\rm (1)} \ One has ${\rm FPGFd} = {\rm FGpd} = {\rm Fpd}$.

\vskip5pt

{\rm (2)} \ If ${\rm fd}M<\infty$, then ${\rm Gfd}M = {\rm fd}M$.

\vskip5pt

{\rm (3)} \ One has ${\rm FGfd} = {\rm Ffd} \le {\rm FPGFd}= {\rm FGpd} = {\rm Fpd}.$

\end{prop}
\begin{proof} (1) \ \cite[Theorem 2.28]{Ho} claims that ${\rm FGpd} = {\rm Fpd}$. Thus, it suffices to prove ${\rm FPGFd} = {\rm Fpd}$.

\vskip5pt

Since $\mathcal P \subseteq \mathcal {PGF}\subseteq \mathcal{GP}$, it follows that ${\rm Gpd}M \le {\rm PGFd}M\le {\rm pd}M$ for any $R$-module.

\vskip5pt

If ${\rm pd}M < \infty$, then ${\rm Gpd}M = {\rm pd}M$ (see \cite[Proposition 2.27]{Ho}), and hence  ${\rm PGFd}M = {\rm pd}M$ (see also \cite[Corollary 3.7(i)]{ET}).
It follows that ${\rm Fpd}\le {\rm FPGFd}.$

\vskip5pt

If $0< {\rm PGFd}M = t < \infty$, then by \cite[Proposition 3.1]{DE} (cf. Lemma \ref{PGFbasic}) there is an exact sequence $0\longrightarrow K\longrightarrow G\longrightarrow M\longrightarrow 0$ with $G\in\mathcal {PGF}$ and
${\rm pd} K = t-1.$ This implies that ${\rm FPGFd}\le {\rm Fpd}+1$.

\vskip5pt

It remains to prove ${\rm FPGFd}\le {\rm Fpd}.$ It suffices to show ${\rm FPGFd}\le {\rm Fpd}$ under the assumption ${\rm FPGFpd} = m<\infty.$ (Otherwise ${\rm FPGFpd} =\infty,$
then by ${\rm FPGFd}\le {\rm Fpd}+1$ one sees that ${\rm Fpd} = \infty$.) Take a module $M$ with ${\rm PGFd}M =m$, it suffices to find a module $L$ with ${\rm pd} L = m$. Thus one can assume that $m\ge 1$.
By \cite[Proposition 3.1]{DE} (cf. Lemma \ref{PGFbasic}) there is an exact sequence $0\longrightarrow K\longrightarrow G\longrightarrow M\longrightarrow 0$ with $G\in\mathcal {PGF}$ and
${\rm pd} K = m-1.$ Since $G\in\mathcal {PGF}$, $G$ is a submodule of a projective module $P$. Put $L = P/K$. Then ${\rm pd} L \le m$.
We claim that ${\rm pd} L = m$, and then we are done.

\vskip5pt

If $m = 1$ then $K$ is projective. If $L$ is projective, then
$0\longrightarrow K \longrightarrow P \longrightarrow L\longrightarrow 0$ splits, and hence $0\longrightarrow K \longrightarrow G \longrightarrow M\longrightarrow 0$ splits. Thus
$M\in\mathcal {PGF}$, i.e., ${\rm PGFd}M =0$. This contradicts  the assumption ${\rm PGFd}M =m = 1$. This proves ${\rm pd} L = 1 = m$.

\vskip5pt If $m \ge 2$, then by ${\rm pd} K = m-1\ge 1$ there is a module $Z$ such that $\Ext^{m-1}_R(K, Z) \ne 0$. By the exact sequence
$\Ext^{m-1}_R(P, Z) = 0 \longrightarrow \Ext^{m-1}_R(K, Z)\longrightarrow \Ext^{m}_R(L, Z)\longrightarrow 0=\Ext^{m}_R(P, Z)$
one sees that $\Ext^{m}_R(L, Z)\ne 0$. Thus ${\rm pd} L = m$.

\vskip5pt

(2) \ According to \cite[Proposition 3.8]{E2}, if $M\in \mathcal{PGF}^\perp $ then ${\rm Gfd}M = {\rm fd}M$. Thus, it suffices to say that ${\mathcal F}^{< \infty}\subseteq \mathcal{PGF}^{\perp}$. By \cite[Theorem 4.4 or Theorem 4.11]{SS2}, ${\mathcal F}\subseteq {\rm PGF}^{\perp}$. Then ${\mathcal F}^{< \infty}\subseteq \mathcal{PGF}^{\perp}$ since $\mathcal{PGF}^{\perp}$ is thick.

\vskip5pt

(3) \ I. Emmanouil \cite[Proposition 2.3]{E2} already claims that ${\rm FGfd} = {\rm Ffd}$.

\vskip5pt

By C. U.  Jensen \cite [Proposition 6]{J}, if ${\rm Fpd}< \infty$ then all flat modules have finite projective dimension. 
It follows that 
${\rm Ffd}\le \mathcal{\rm Fpd}$.

\vskip5pt

Together with (1) one gets
${\rm FGfd} = {\rm Ffd} \le {\rm FPGFd}= {\rm FGpd} = {\rm Fpd}.$
\end{proof}

\begin{rem} We thank Ioannis Emmanouil for pointing out an error in the original proof of Proposition \ref{finitisticdims}(3), and for pointing out to us the paper \cite{J}.
\end{rem}

\subsection{Abelian model structures arising from $\mathcal {PGF}^{<\infty}$, $\mathcal {GP}^{<\infty}$, and $\mathcal {GF}^{<\infty}$}

\begin{cor} \label{finitedimensionversion} \ Let $R$ be a ring. Assume that the finitistic projective dimension ${\rm Fpd}$ of $R$ is finite, say  ${\rm Fpd}=t$. Then

\vskip5pt

{\rm (1)} \ The pair \ $(\mathcal{PGF}^{<\infty}, \  (\mathcal{P}^{<\infty})^{\perp} \cap \mathcal {PGF}^{\perp}) = (\mathcal{PGF}_t, \  \mathcal{P}_t^{\perp} \cap \mathcal {PGF}^{\perp})$
is a complete and hereditary cotorsion pair, and \ $(\mathcal{PGF}^{<\infty}, \ (\mathcal{P}^{<\infty})^{\perp}, \ \mathcal {PGF}^{\perp})
= (\mathcal{PGF}_t, \ \mathcal{P}_t^{\perp}, \ \mathcal {PGF}^{\perp})$ is a hereditary Hovey triple,  in $R\mbox{-}{\rm Mod}.$

\vskip5pt

{\rm (2)} \ If $R$ is an Artin algebra, then  \ $(\mathcal{GP}^{<\infty}, \  (\mathcal{P}^{<\infty})^{\perp} \cap \mathcal {GP}^{\perp}) =
(\mathcal{GP}_t, \ \mathcal{P}_t^{\perp} \cap \mathcal {GP}^{\perp})$ is a complete and hereditary cotorsion pair,
and \ $(\mathcal{GP}^{<\infty}, \ (\mathcal{P}^{<\infty})^{\perp}, \ \mathcal {GP}^{\perp}) = (\mathcal{GP}_t, \ \mathcal{P}_t^{\perp}, \ \mathcal {GP}^{\perp})$
is a hereditary Hovey triple,  in $R\mbox{-}{\rm Mod}.$

\vskip5pt

{\rm (3)} \ The pair
\ $(\mathcal{GF}^{<\infty}, \  (\mathcal{F}^{<\infty})^{\perp} \cap \mathcal {PGF}^{\perp}) = (\mathcal{GF}_t, \  \mathcal{F}_t^{\perp} \cap \mathcal {PGF}^{\perp})$
is a complete and hereditary cotorsion pair, and
\ $(\mathcal{GF}^{<\infty}, \ (\mathcal{F}^{<\infty})^{\perp}, \ \mathcal {PGF}^{\perp}) = (\mathcal{GF}_t, \ \mathcal{F}_t^{\perp}, \ \mathcal {PGF}^{\perp})$ is a hereditary Hovey triple, in $R\mbox{-}{\rm Mod}.$
\end{cor}

\begin{proof} \ By Proposition \ref{finitisticdims}(3) one has ${\rm FGFd}= {\rm Ffd}\le {\rm FPGFd} = {\rm FGpd} = t$. Thus
$\mathcal{PGF}^{<\infty} = \mathcal{PGF}_t$ and $\mathcal{P}^{<\infty} = \mathcal{P}_t$, \ $\mathcal{GP}^{<\infty} = \mathcal{GP}_t$ \ and $\mathcal{P}^{<\infty} = \mathcal{P}_t$,
\ and $\mathcal{GF}^{<\infty} = \mathcal{GF}_t$ \ and $\mathcal{F}^{<\infty} = \mathcal{F}_t$.
 Then the assertions (1), (2) and (3) follow from Corollary \ref{mainthm2}, Theorem \ref{mainthm4}, and \cite[Theorem A]{M}, respectively.\end{proof}

\section{\bf Induced complete cotorsion pairs in exact categories with applications}
\subsection{\bf Induced complete cotorsion pairs in exact categories}

The following observation will provide new complete cotorsion pairs in exact categories.

\vskip5pt

\begin{thm} \label{inducedctp1} \ Let $\mathcal A$ be an abelian category, and $\mathcal B$ a full subcategory which is closed under extensions.
Suppose that $(\mathcal X, \ \mathcal Y)$ is a complete cotorsion pair in $\mathcal A$ and $\mathcal X\subseteq \mathcal B$. If $\mathcal B$ is closed under the kernels of epimorphisms, then

\vskip5pt

$(1)$ \  $(\mathcal X, \ \mathcal Y\cap\mathcal B)$ is a complete cotorsion pair in exact category $\mathcal B.$

\vskip5pt

$(2)$ \  If  $(\mathcal X, \ \mathcal Y)$ is a hereditary cotorsion pair in $\mathcal A$, then $(\mathcal X, \ \mathcal Y\cap\mathcal B)$ is a hereditary  cotorsion pair in $\mathcal B$.
\end{thm}

\begin{proof} \ $(1)$ \ It is clear that $\mathcal B$ is an exact category, and $\Ext^1_\mathcal B(X,  Y) = \Ext^1_\mathcal A(X,  Y) = 0$
for $X\in \mathcal X$ and $Y\in \mathcal Y\cap\mathcal B$. Thus
$\mathcal Y\cap\mathcal B\subseteq \mathcal X^\perp\cap\mathcal B$ and $\mathcal X\subseteq \ ^\perp(\mathcal Y\cap\mathcal B)\cap\mathcal B$.
Let $M\in \mathcal X^\perp\cap\mathcal B$. Then it is clear that $M\in \mathcal Y\cap\mathcal B$. Thus
$\mathcal Y\cap\mathcal B = \mathcal X^\perp\cap\mathcal B$. Let $M\in \ ^\perp(\mathcal Y\cap\mathcal B)\cap\mathcal B$. Since
$(\mathcal X, \ \mathcal Y)$ is a complete cotorsion pair in $\mathcal A$, there is an exact sequence
$0\longrightarrow Y\longrightarrow  X\longrightarrow M\longrightarrow 0$ with $X\in\mathcal X$ and $Y\in\mathcal Y$. Since $\mathcal X\subseteq \mathcal B$
and $\mathcal B$ is closed under the kernels of epimorphisms, $Y\in \mathcal Y\cap \mathcal B.$ Since $M\in \ ^\perp(\mathcal Y\cap\mathcal B)\cap\mathcal B$, \ $\Ext^1_\mathcal A(M, Y) = 0$,
thus the short exact sequence splits and $M\in\mathcal X.$

\vskip5pt

Thus $(\mathcal X, \ \mathcal Y\cap\mathcal B)$ is a cotorsion pair in $\mathcal B.$ For $M\in\mathcal B$, as above there is an exact sequence
$0\longrightarrow Y\longrightarrow  X\longrightarrow M\longrightarrow 0$ with $X\in\mathcal X$ and $Y\in\mathcal Y\cap \mathcal B$. By definition it is an admissible exact sequence in $\mathcal B.$  Also there is an exact sequence
$0\longrightarrow M\longrightarrow  Y'\longrightarrow X'\longrightarrow 0$ with $Y'\in\mathcal Y$ and $X'\in\mathcal X \subseteq \mathcal B$. Then
$Y'\in \mathcal Y\cap\mathcal B.$ Thus $(\mathcal X, \ \mathcal Y\cap\mathcal B)$ is a complete cotorsion pair in $\mathcal B.$

\vskip5pt

$(2)$ \ Since $\mathcal X$ is closed under the kernels of epimorphisms in $\mathcal A$, $\mathcal X$ is closed under the kernels of deflations
in $\mathcal B$. Since $(\mathcal X, \ \mathcal Y\cap\mathcal B)$ is a complete cotorsion pair in $\mathcal B,$ it follows that it is also hereditary (cf. \cite[Lemma 6.17]{S}; or Subsection 1.2).
 \end{proof}

\vskip5pt

The dual of Theorem \ref{inducedctp1} is

\vskip5pt

{\bf Theorem 5.1'} \label{inducedctp2} \ {\it Let $\mathcal A$ be an abelian category, and $\mathcal B$ an full subcategory which is closed under extensions.
Suppose that $(\mathcal X, \ \mathcal Y)$ is a complete cotorsion pair in $\mathcal A$ and $\mathcal Y\subseteq \mathcal B$. If $\mathcal B$ is closed under the cokernels of monomorphisms, then

\vskip5pt

$(1)$ \  $(\mathcal X\cap\mathcal B, \ \mathcal Y)$ is a complete cotorsion pair in exact category $\mathcal B.$

\vskip5pt

$(2)$ \  If  $(\mathcal X, \ \mathcal Y)$ is a hereditary cotorsion pair in $\mathcal A$, then $(\mathcal X\cap\mathcal B, \ \mathcal Y)$ is a hereditary  cotorsion pair in $\mathcal B$.}

\vskip5pt

\subsection{Application 1: Exact model structures on $\mathcal{PGF}_n$} For each non-negative integer $n$, one has already known
two chains of complete and hereditary cotorsion pairs $(\mathcal X, \ \mathcal Y)$ in $R$-Mod with
$\mathcal X\subseteq \mathcal{PGF}_n$, namely $(\mathcal P_m, \ \mathcal P_m^\perp), \  \ (\mathcal {PGF}_m, \ \mathcal P_m^\perp\cap \mathcal {PGF}^\perp)$,
for each non-negative integer $m$ with $m\le n$.  See Lemma \ref{ctp7.4.6} and Corollary \ref{mainthm2}. Applying Theorem \ref{inducedctp1} one will get
new chains of complete and hereditary cotorsion pairs, and exact model structures, on $\mathcal{PGF}_n$.

\vskip5pt

Note that $\mathcal P_n$ is not thick in $R$-Mod, in general.

\begin{lem} \label{pnthickinpgfn} \ Let $R$ be a ring and $n$ a non-negative integer. Then
$\mathcal P_n$ is thick in $\mathcal {PGF}_n$ and in $\mathcal{G P}_n$.
\end{lem}
\begin{proof} \ We only prove that $\mathcal P_n$ is thick in $\mathcal {PGF}_n$.
Let $0\longrightarrow X\longrightarrow Y\longrightarrow Z\longrightarrow 0$ be an exact sequence in $\mathcal {PGF}_n$.
It suffices to prove that if $X\in \mathcal P_n$ and $Y\in \mathcal P_n$ then $Z\in \mathcal P_n$.
By $\mathcal {PGF}^{\perp}\cap\mathcal{PGF}_n  = \mathcal{P}_n$ (cf. Theorem \ref{mainthm1})
one has $Z\in \mathcal P_{n+1}\cap \mathcal {PGF}_n = \mathcal {PGF}^{\perp}\cap\mathcal{PGF}_{n+1} \cap \mathcal {PGF}_n
= \mathcal {PGF}^{\perp}\cap \mathcal {PGF}_n = \mathcal P_n.$

\vskip5pt

{\bf An alternative proof:} \ Since $Z\in \mathcal {PGF}_n$, by Lemma \ref{PGFbasic}
there is an exact sequence $0\longrightarrow K\longrightarrow G\longrightarrow Z\longrightarrow 0$ with $G\in \mathcal {PGF}$ and $K\in\mathcal P_{n-1}$. (If $n=0$ then $K = 0$.)
The following proof holds also for $n = 0$. Taking the pullback

\[
\xymatrix@R=0.6cm@C=30pt{& & 0\ar[d]  & 0\ar[d] \\ & & K\ar[d]\ar@{=}[r]  & K\ar[d]\\
0\ar[r] & X\ar[r]\ar@{=}[d]  & L\ar@{.>}[r]\ar@{.>}[d]  &  G\ar[r]\ar[d] & 0\\
0\ar[r] & X \ar[r]  & Y\ar[r]\ar[d]  & Z\ar[r]\ar[d] & 0\\
& &0 & 0}
\]

\vskip5pt

\noindent Since $K\in\mathcal P_{n-1}$ and $Y\in\mathcal P_n$, it follows that $L\in\mathcal P_n$.
Since $X\in\mathcal P_n$ and $L\in\mathcal P_n$, it follows that ${\rm pd} G< \infty.$ While $G\in \mathcal {PGF}$, $G$ is Gorenstein projective, and hence $G$ has to be projective (cf. \cite[Proposition 10.2.3]{EJ}).
Thus $Z\in\mathcal P_n$.
\end{proof}

\begin{cor} \label{mainthm5} \ Let $R$ be a ring, $m$ and $n$ non-negative integers with $m\le n$. Then

\vskip5pt

{\rm (1)} \ The pair \ $(\mathcal{P}_m, \ \mathcal P_m^\perp\cap \mathcal{PGF}_n)$ is a complete and hereditary cotorsion pair in weakly idempotent complete exact category $\mathcal{PGF}_n$, with
kernel $\mathcal{P}_m\cap\mathcal{P}_m^{\perp}$.

\vskip5pt

{\rm (2)} \ The pair \ $(\mathcal{PGF}_m, \ \mathcal P_m^\perp\cap \mathcal{P}_n)$ is a complete and hereditary cotorsion pair in  $\mathcal{PGF}_n$, with
kernel $\mathcal{P}_m\cap\mathcal{P}_m^{\perp}$.

\vskip 5pt

{\rm (3)}  \ The triple \ $(\mathcal{PGF}_m, \ \mathcal P_m^\perp\cap\mathcal{PGF}_n, \ \mathcal P_{n})$ is a hereditary Hovey triple in $\mathcal{PGF}_n$.
In particular, $\mathcal{PGF}_m\cap \mathcal P_m^\perp$ is a Frobenius category with $\mathcal P_{m}\cap \mathcal P_m^\perp$ as the class of projective-injective objects,
and the corresponding homotopy category is \
$(\mathcal{PGF}_m\cap \mathcal P_m^\perp)/(\mathcal P_{m}\cap \mathcal P_m^\perp)\cong \mathcal{PGF}/\mathcal P.$
\end{cor}

\begin{proof} \ Note that $\mathcal{PGF}_n$ is closed under extensions and the kernels of epimorphisms (cf. Lemma \ref{PGFdim}) and
closed under direct summands (cf. \cite[Proposition 2.3]{DE}). In particular, $\mathcal{PGF}_n$ is a weakly idempotent complete exact category.

\vskip5pt

(1) \  By Lemma \ref{ctp7.4.6}, $(\mathcal{P}_m, \ \mathcal P_m^\perp)$ is a complete and hereditary cotorsion pair in $R$-Mod, for each non-negative integer $m$.
Since $m\le n$, $\mathcal{P}_m\subseteq \mathcal{P}_n\subseteq \mathcal{PGF}_n.$
Then the assertion follows from Theorem \ref{inducedctp1}.

\vskip5pt

(2) \  By Corollary \ref{mainthm2}, $(\mathcal{PGF}_m, \ \mathcal P_m^\perp\cap\mathcal {PGF}^{\perp})$ is a complete and hereditary cotorsion pair in $R$-Mod, for each non-negative integer $m$.
Since $m\le n$,  $\mathcal{PGF}_m\subseteq \mathcal{PGF}_n.$ Applying Theorem \ref{inducedctp1}, one sees that $(\mathcal{PGF}_m, \ \mathcal P_m^\perp\cap\mathcal {PGF}^{\perp}\cap \mathcal{PGF}_n)$ is a complete and hereditary cotorsion pair in $\mathcal{PGF}_n$. By Theorem \ref{mainthm1},  $\mathcal {PGF}^{\perp}\cap \mathcal{PGF}_n = \mathcal P_n$, thus
this cotorsion pair is just \ $(\mathcal{PGF}_m, \ \mathcal P_m^\perp\cap \mathcal{P}_n).$

\vskip 5pt

Note that $\mathcal{PGF}_m\cap \mathcal{P}_n = \mathcal{P}_m.$ (This follows from $\mathcal {PGF}^{\perp}\cap \mathcal{PGF}_n = \mathcal P_n$.)
Thus the kernel of this cotorsion pair is $\mathcal{P}_m\cap\mathcal{P}_m^{\perp}$.
\vskip5pt

(3) \ Put $(\mathcal C, \ \mathcal F, \ \mathcal W) = (\mathcal{PGF}_m, \ \mathcal P_m^\perp\cap\mathcal{PGF}_n, \ \mathcal P_{n})$.
By Lemma \ref{pnthickinpgfn}, $\mathcal W = \mathcal P_n$ is thick in $\mathcal{PGF}_n$.

\vskip 5pt

Since $\mathcal C\cap \mathcal W = \mathcal{PGF}_m\cap \mathcal P_{n} = \mathcal P_{m}$, it follows from (1) that
$(\mathcal C\cap \mathcal W, \ \mathcal F) = (\mathcal{P}_m, \ \mathcal P_m^\perp\cap\mathcal{PGF}_n)$ is a complete and hereditary
cotorsion pairs in $\mathcal{PGF}_n$.
Also, $(\mathcal C, \ \mathcal F\cap \mathcal W) = (\mathcal{PGF}_m, \ \mathcal P_m^\perp\cap\mathcal P_{n})$ is a complete and hereditary
cotorsion pairs in $\mathcal{PGF}_n$, by (2). So the triple is a hereditary Hovey triple in $\mathcal{PGF}_n$.
By Theorem \ref{hoherhov}, $\mathcal C\cap \mathcal F = \mathcal{PGF}_m\cap \mathcal P_m^\perp$ is a Frobenius category,
with $$\mathcal C\cap\mathcal F\cap\mathcal W = \mathcal{PGF}_m\cap \mathcal P_m^\perp\cap \mathcal P_{n}= \mathcal P_m\cap \mathcal P_{m}^\perp$$ as the class of projective-injective objects,
and the homotopy category is
$(\mathcal C\cap\mathcal F)/(\mathcal C\cap\mathcal F\cap\mathcal W)
= (\mathcal{PGF}_m\cap \mathcal P_m^\perp)/(\mathcal P_{m}\cap \mathcal P_m^\perp),$ which is triangle equivalent to $\mathcal{PGF}/\mathcal P$, by Corollary \ref{wdeterminetriho}.
\end{proof}

\vskip5pt

As we see, the exact model structure on $\mathcal{PGF}_n$ given in Corollary \ref{mainthm5}(3) is not trivial, and in general it is not projective.

\vskip5pt

\subsection{Application 2: Exact model structures on $\mathcal{PGF}^{<\infty}$}
It is clear that $\mathcal P_m\cap \mathcal{PGF}^{<\infty} = \mathcal P_{m}$. Also note that $\mathcal{PGF}_m\cap \mathcal P^{<\infty} = \mathcal P_{m}$.
In fact, by $\mathcal {PGF}^{\perp}\cap \mathcal{PGF}_t = \mathcal P_t$ for any non-negative integer (cf. Theorem \ref{mainthm1}) one has
\begin{align*}\mathcal{PGF}_m\cap \mathcal P^{<\infty}& = \bigcup\limits_{t\ge 0}(\mathcal{PGF}_m\cap \mathcal P_t)
= \bigcup\limits_{t\ge 0}(\mathcal{PGF}_m\cap \mathcal {PGF}^{\perp}\cap \mathcal{PGF}_t) = \bigcup\limits_{t\ge 0}(\mathcal{P}_m\cap \mathcal{PGF}_t)\\& =
\mathcal P_m\cap \mathcal {PGF}^{<\infty} = \mathcal P_{m}.\end{align*}

\vskip5pt

By the similar argument as in Corollary \ref{mainthm5} one gets

\vskip5pt

\begin{cor} \label{mainthm5.5} \ Let $R$ be a ring and $m$ a non-negative integer. Then

\vskip5pt

{\rm (1)} \ The pair \ $(\mathcal{P}_m, \ \mathcal P_m^\perp\cap \mathcal{PGF}^{<\infty})$ is a complete and hereditary cotorsion pair
in weakly idempotent complete exact category $\mathcal{PGF}^{<\infty}$,
with kernel $\mathcal{P}_m\cap\mathcal{P}_m^{\perp}$.

\vskip5pt

{\rm (2)} \ The pair \ $(\mathcal{PGF}_m, \ \mathcal P_m^\perp\cap \mathcal{P}^{<\infty})$ is a complete and hereditary cotorsion pair in $\mathcal{PGF}^{<\infty}$
with kernel $\mathcal{P}_m\cap\mathcal{P}_m^{\perp}$.

\vskip 5pt

{\rm (3)}  \ The triple \ $(\mathcal{PGF}_m, \ \mathcal P_m^\perp\cap\mathcal{PGF}^{<\infty}, \ \mathcal P^{<\infty})$ is a hereditary Hovey triple in $\mathcal{PGF}^{<\infty}$.
The corresponding homotopy category is \
$(\mathcal{PGF}_m\cap \mathcal P_m^\perp)/(\mathcal P_{m}\cap \mathcal P_m^\perp)\cong \mathcal{PGF}/\mathcal P.$
\end{cor}

When $m = 0$, Corollary \ref{mainthm5.5} has been obtained by Dalezios and Emmanouil {\rm\cite[Theorem 4.1]{DE}.}

\section{\bf Exact model structures on $\mathcal{G P}_n$ and $\mathcal{G P}^{<\infty}$}

\subsection{Application 3: Exact model structures on $\mathcal{G P}_n$} For any non-negative integer $n$, applying Theorem \ref{inducedctp1} to $\mathcal B = \mathcal{GP}_n$ and the following
two chains of complete and hereditary cotorsion pairs $(\mathcal X, \ \mathcal Y)$ in $R$-Mod with
$\mathcal X\subseteq \mathcal B = \mathcal{GP}_n$: $$(\mathcal P_m, \ \mathcal P_m^\perp), \ \ \ \ (\mathcal {PGF}_m, \ \mathcal P_m^\perp\cap \mathcal {PGF}^\perp),$$
where $0\le m\le n$ (see Lemma \ref{ctp7.4.6} and Corollary \ref{mainthm2}), one gets

\begin{cor} \label{mainthm6} \ Let $R$ be a ring,  $m$ and $n$ non-negative integers with $m\le n$. Then

\vskip5pt

{\rm (1)} \ The pair \ $(\mathcal{P}_m, \ \mathcal P_m^\perp\cap \mathcal{GP}_n)$ is a complete and hereditary cotorsion pair in weakly idempotent complete exact category $\mathcal{GP}_n$, with
kernel $\mathcal{P}_m\cap\mathcal{P}_m^{\perp}$.

\vskip5pt

{\rm (2)} \ The pair \ $(\mathcal{PGF}_m, \ \mathcal P_m^\perp\cap \mathcal {PGF}^\perp\cap \mathcal{GP}_n)$ is a complete and hereditary cotorsion pair in  $\mathcal{GP}_n$, with
kernel $\mathcal{P}_m\cap\mathcal{P}_m^{\perp}$.

\vskip5pt

{\rm (3)}  \ The triple \ $(\mathcal{PGF}_m, \ \mathcal P_m^\perp\cap\mathcal{G P}_n, \ \mathcal{PGF}^\perp\cap \mathcal{GP}_n)$ is a hereditary Hovey triple in $\mathcal{G P}_n$.
The corresponding homotopy category is
\ $(\mathcal{PGF}_m\cap \mathcal P_m^\perp) /(\mathcal{P}_m\cap \mathcal P_m^\perp)\cong \mathcal{PGF}/\mathcal{P}.$
\end{cor}

\begin{proof} \ Note that $\mathcal{G P}_n$ is closed under extensions, the kernels of epimorphisms,  and direct summands (cf. \cite[Theorem 2.20, Proposition 2.19]{Ho}).
In particular, $\mathcal{G P}_n$ is a weakly idempotent complete exact category.

\vskip 5pt

(1) \ Applying Theorem \ref{inducedctp1} to $\mathcal B = \mathcal{GP}_n$ and the complete and hereditary cotorsion pair $(\mathcal P_m, \ \mathcal P_m^\perp)$ in $R$-Mod,
one gets the assertion.

\vskip 5pt

(2) \ Applying Theorem \ref{inducedctp1} to $\mathcal{GP}_n$ and the complete and hereditary cotorsion pair $(\mathcal {PGF}_m, \ \mathcal P_m^\perp\cap \mathcal {PGF}^\perp)$ in $R$-Mod, one gets
the assertion. Since $\mathcal{PGF}_m \cap\mathcal{PGF}^\perp = \mathcal P_m$ (cf. Theorem \ref{mainthm1}), the kernel of this cotorsion pair is
$\mathcal{P}_m\cap\mathcal P_m^\perp$.

\vskip 5pt

(3) \ Since $\mathcal{PGF}^\perp$ is thick in $R$-Mod (cf. Lemma \ref{ctpPGF}),  $\mathcal{PGF}^\perp\cap \mathcal{GP}_n$ is thick in $\mathcal{GP}_n$.
By (1) and (2), \ $(\mathcal{PGF}_m, \ \mathcal P_m^\perp\cap\mathcal{G P}_n, \ \mathcal{PGF}^\perp\cap \mathcal{GP}_n)$ is a hereditary Hovey triple in $\mathcal{G P}_n$.
The last assertion follows from Theorem \ref{hoherhov} and Corollary \ref{wdeterminetriho}, by using that
$(\mathcal{PGF}, \ \mathcal{G P}_n, \ \mathcal{PGF}^\perp\cap \mathcal{GP}_n)$ is also a hereditary Hovey triple in $\mathcal{GP}_n$.
\end{proof}

\subsection{A new chain of cotorsion pairs and exact model structures on $\mathcal{G P}_n$}
It is quite surprising that for $0\le m\leq n$, \ $(\mathcal{GP}_m,\mathcal{P}_m^\perp\cap\mathcal{GP}^\perp\cap \mathcal{GP}_n)=(\mathcal{GP}_m,\mathcal{P}_m^\perp\cap \mathcal{P}_n)$
is always a complete and hereditary cotorsion pair in $\mathcal{GP}_n$ (see Theorem \ref{gpctp}), since in general
we don't know whether $(\mathcal{GP}_m,\mathcal{P}_m^\perp\cap\mathcal{GP}^\perp)$ is a cotorsion pair in $R\mbox{-}{\rm Mod}$ for an arbitrary ring $R$. Thus, Theorem \ref{gpctp}
is not an application of Theorem \ref{inducedctp1} and Theorem \ref{mainthm4}. (Note that Theorem \ref{mainthm4} needs the assumption of Artin algebra.)

\vskip 5pt

\begin{lem}\label{lemgp} \ Let $R$ be a ring. Then  $^\perp\left(\mathcal P_m^\perp\cap\mathcal{P}_n\right)\cap \mathcal{P}_n= \mathcal{P}_m$ in $R\mbox{-}{\rm Mod}$ for $0\leq m\leq n$.
\end{lem}
\begin{proof}
Clearly $\mathcal{P}_m\subseteq \ ^\perp\left(\mathcal P_m^\perp\cap\mathcal{P}_n\right)\cap \mathcal{P}_n$. For the other hand, let $M\in \ ^\perp\left(\mathcal P_m^\perp\cap\mathcal{P}_n\right)\cap \mathcal{P}_n$. Since \ $(\mathcal{P}_m, \ \mathcal{P}_m^{\perp})$ is a complete cotorsion pair (cf. Lemma \ref{ctp7.4.6}),
there is an exact sequence  $0\longrightarrow N \longrightarrow H\longrightarrow  M\longrightarrow 0$
with  $H\in \mathcal{P}_m$ and $N\in \mathcal{P}_m^{\perp}$. Since $\mathcal{P}_n$ is closed under kernels of epimorphisms, $N\in \mathcal{P}_n$. Thus $N\in \mathcal P_m^\perp\cap\mathcal{P}_n$ and this exact sequence splits. It follows that $M\in \mathcal{P}_m$.
\end{proof}

\begin{thm}\label{gpctp}  \ Let $R$ be a ring, $m$ and $n$ non-negative integers with $m\le n$. Then

\vskip5pt

{\rm (1)} \ The pair $(\mathcal{GP}_m, \ \mathcal P_m^\perp\cap\mathcal{P}_n)$ is a complete and hereditary cotorsion pair in $\mathcal{G P}_n$, with
kernel $\mathcal{P}_m\cap\mathcal{P}_m^{\perp}$.

\vskip5pt

{\rm (2)} \ The triple \ $(\mathcal{G P}_m, \ \mathcal P_m^\perp\cap\mathcal{G P}_n, \ \mathcal P_{n})$ is a hereditary Hovey triple in $\mathcal{G P}_n$.
The corresponding homotopy category is
\ $(\mathcal{G P}_m\cap \mathcal P_m^\perp) /(\mathcal{P}_m\cap \mathcal P_m^\perp)\cong \mathcal{G P}/\mathcal{P}.$
\end{thm}

\begin{proof} \ {\rm (1)} \ To prove that $(\mathcal{GP}_m, \ \mathcal P_m^\perp\cap\mathcal{P}_n)$ is a cotorsion pair in $\mathcal{G P}_n$, by definition it suffices to prove that  $\mathcal{GP}_m^\perp \cap \mathcal{G P}_n=\mathcal P_m^\perp\cap\mathcal{P}_n$ and \ $^\perp\left(\mathcal P_m^\perp\cap\mathcal{P}_n\right)\cap \mathcal{G P}_n= \mathcal{GP}_m$.

\vskip5pt

Since $\mathcal P_m^\perp\cap\mathcal{P}_n\subseteq\mathcal{P}_m^{\perp}\cap \mathcal {GP}^{\perp}$, it follows from Theorem \ref{mainthm3}(3) that
$${\rm Ext}^{1}_{\mathcal{GP}_n}(\mathcal{GP}_m, \ \mathcal P_m^\perp\cap\mathcal{P}_n)={\rm Ext}^{1}_{R}(\mathcal{GP}_m, \ \mathcal P_m^\perp\cap\mathcal{P}_n)=0.$$ It follows that $ \mathcal P_m^\perp\cap\mathcal{P}_n\subseteq\mathcal{GP}_m^\perp\cap \mathcal{G P}_n $ and $\mathcal{GP}_m\subseteq\ ^\perp\left(\mathcal P_m^\perp\cap\mathcal{P}_n\right)\cap \mathcal{G P}_n$.

\vskip5pt

 Since $\mathcal{GP}_m^\perp\subseteq \mathcal{P}_m^\perp\cap \mathcal{GP}^\perp$ and $\mathcal{GP}^\perp \cap \mathcal{G P}_n=\mathcal{P}_n$(cf. Theorem \ref{mainthm3}), it follows that $\mathcal{GP}_m^\perp \cap \mathcal{G P}_n\subseteq \mathcal{P}_m^\perp \cap\mathcal{GP}^\perp \cap \mathcal{G P}_n= \mathcal P_m^\perp\cap\mathcal{P}_n$, and hence $\mathcal{GP}_m^\perp \cap \mathcal{G P}_n=\mathcal P_m^\perp\cap\mathcal{P}_n$.

 \vskip5pt

It remains to prove $^\perp\left(\mathcal P_m^\perp\cap\mathcal{P}_n\right)\cap \mathcal{G P}_n \subseteq \mathcal{GP}_m$. Thus one can assume that $n\ge 1$. Let $M\in \ ^\perp\left(\mathcal P_m^\perp\cap\mathcal{P}_n\right)\cap \mathcal{G P}_n$. By Lemma \ref{basic} there is an admissible exact sequence $0\longrightarrow K \longrightarrow G \longrightarrow M \longrightarrow 0$
with $G \in \mathcal{GP}$ and $K\in \mathcal{P}_{n-1}$. Since $G$ is Gorenstein projective, by definition there is
an admissible exact sequence $0\longrightarrow G \longrightarrow P \longrightarrow G^{\prime} \longrightarrow 0$ with $P\in\mathcal P$  and $G^\prime \in \mathcal{GP}$.
Take the pushout

\[
\xymatrix@C=30pt@R=0.6cm{
  & & 0\ar[d]  & 0\ar[d] \\
0\ar[r] & K\ar[r]\ar@{=}[d]  & G\ar[r]\ar[d]  &  M\ar[r]\ar@{.>}[d] & 0\\
0\ar[r] & K \ar[r]  & P\ar@{.>}[r]\ar[d]  & H\ar[r]\ar[d] & 0\\
& &G^\prime \ar[d]\ar@{=}[r] & G^\prime\ar[d]\\
& & 0 &0 }
\]
Then $H \in \mathcal{P}_n$. Since $G^\prime \in \mathcal{GP}\subseteq\ ^\perp\left(\mathcal P_m^\perp\cap\mathcal{P}_n\right)$ and the class $^\perp\left(\mathcal P_m^\perp\cap\mathcal{P}_n\right)$ is closed under extensions, it follows that $H\in\ ^\perp\left(\mathcal P_m^\perp\cap\mathcal{P}_n\right)$. Thus $H\in\ ^\perp\left(\mathcal P_m^\perp\cap\mathcal{P}_n\right)\cap \mathcal{P}_n=\mathcal{P}_m$ (cf. Lemma \ref{lemgp}), and hence $M\in \mathcal{GP}_m$
by Theorem \ref{mainthm3} $(4)\Longrightarrow (1)$.

\vskip5pt

Up to now we have proved that $(\mathcal{GP}_m, \ \mathcal P_m^\perp\cap\mathcal{P}_n)$ is a cotorsion pair in $\mathcal{G P}_n$.
Now we prove that it is complete.
By Lemma \ref{basic}, there is an admissible exact sequence
$$0 \longrightarrow N \longrightarrow H \longrightarrow M \longrightarrow 0$$ with $H\in \mathcal{GP}$ and $N\in \mathcal{P}_{n-1}$. (If $n = 0$ then $K = 0$.)
Since $(\mathcal{P}_m, \  \mathcal{P}_m^{\perp}\cap \mathcal{GP}_n)$ is a complete cotorsion pair in $\mathcal{GP}_n$, there is an admissible exact sequence
$$0 \longrightarrow N \longrightarrow L \longrightarrow C \longrightarrow 0$$ with $L\in \mathcal{P}_m^{\perp}\cap \mathcal{GP}_n$ and $C\in  \mathcal {P}_m$.
Take the pushout

\[
\xymatrix@C=30pt@R=0.6cm{& 0\ar[d]  & 0\ar[d] \\
0\ar[r] & N\ar[r]\ar[d]  & H\ar[r]\ar@{.>}[d]  &  M\ar[r]\ar@{=}[d] & 0\\
0\ar[r] & L \ar@{.>}[r]\ar[d]  & G\ar[r]\ar[d]  & M\ar[r] & 0\\
& C \ar[d]\ar@{=}[r] & C\ar[d]\\
& 0 &0}
\]
Since $N\in \mathcal{P}_n$ and $C\in \mathcal {P}_n$, it follows that $L\in \mathcal{P}_n$, and hence
$L\in \mathcal{P}_m^{\perp} \cap \mathcal {P}_n$. Since $H\in \mathcal{GP}\subseteq \mathcal{G P}_m$  and $C\in \mathcal {P}_m \subseteq \mathcal{G P}_m$, it follows that
$G\in \mathcal{G P}_m$. Thus we get an admissible exact sequence
$$0 \longrightarrow L \longrightarrow G \longrightarrow M \longrightarrow 0$$ with $G\in \mathcal{G P}_m$ and $L\in \mathcal P_m^\perp\cap\mathcal{P}_n$.

\vskip5pt

Now we show that for $M\in \mathcal{GP}_n$, there is an admissible exact sequence
$$0 \longrightarrow M \longrightarrow L \longrightarrow G \longrightarrow 0$$ with $L\in \mathcal P_m^\perp\cap\mathcal{P}_n$ and $G\in \mathcal{G P}_m$.
By Theorem \ref{mainthm3} $(1)\Longrightarrow(4)$, there is an admissible exact sequence
$$0 \longrightarrow M \longrightarrow N \longrightarrow H \longrightarrow 0$$ with $N\in \mathcal{P}_{n}$ and $H\in \mathcal{GP}$. Since $(\mathcal{P}_m, \ \mathcal{P}_m^{\perp}\cap \mathcal{GP}_n)$ is a complete cotorsion pair in $\mathcal{GP}_n$ (cf. Corollary \ref{mainthm6}(1)), there is an admissible exact sequence
$$0 \longrightarrow N \longrightarrow L \longrightarrow C \longrightarrow 0$$ with $L\in \mathcal{P}_m^{\perp}\cap \mathcal{GP}_n$ and $C\in  \mathcal {P}_m$.
Take the pushout
\[
\xymatrix@C=30pt@R=0.6cm{
  & & 0\ar[d]  & 0\ar[d] \\
0\ar[r] & M\ar[r]\ar@{=}[d]  & N\ar[r]\ar[d]  &  H\ar[r]\ar@{.>}[d] & 0\\
0\ar[r] & M \ar[r]  & L\ar@{.>}[r]\ar[d]  & G\ar[r]\ar[d] & 0\\
& &C \ar[d]\ar@{=}[r] & C\ar[d]\\
& & 0 &0 }
\]
Since $N\in \mathcal{P}_n$ and $C\in \mathcal {P}_n$, it follows that $L\in \mathcal{P}_n$, and hence
$L\in \mathcal{P}_m^{\perp} \cap \mathcal {P}_n$. Since $H\in \mathcal{GP}\subseteq \mathcal{G P}_m$  and $C\in \mathcal {P}_m \subseteq \mathcal{G P}_m$, it follows that
$G\in \mathcal{G P}_m$. We are done. That is, $(\mathcal{GP}_m, \ \mathcal P_m^\perp\cap\mathcal{P}_n)$ is a complete cotorsion pair in $\mathcal{G P}_n$.

\vskip 5pt

The heredity of this cotorsion pair follows from the fact that $\mathcal{GP}_m$ is closed under kernels of deflations in $\mathcal{GP}_n$.
Since $\mathcal{GP}_m  \cap\mathcal {P}_n = \mathcal{P}_m,$  it follows that the kernel of this cotorsion pair is  $\mathcal{P}_m\cap\mathcal{P}_m^{\perp}$.

\vskip 5pt

{\rm (2)} \ Since $\mathcal{P}_n$ is thick in $\mathcal{GP}_n$ (cf. Lemma \ref{pnthickinpgfn}),
it follows from (1) and Corollary \ref{mainthm6}(1) that
\ $(\mathcal{G P}_m, \ \mathcal P_m^\perp\cap\mathcal{G P}_n, \ \mathcal P_{n})$ is a hereditary Hovey triple in $\mathcal{G P}_n$.
The last assertion follows from Theorem \ref{hoherhov} and Corollary \ref{wdeterminetriho}, by using that
$(\mathcal{GP}, \ \mathcal{G P}_n, \ \mathcal{P}_n)$ is also a hereditary Hovey triple in $\mathcal{GP}_n$.
\end{proof}

\subsection{A sufficient and necessary condition for $\mathcal{PGF} = \mathcal {GP}$}

\begin{cor} \ Let $R$ be a ring. Then the following are equivalent$:$

\vskip5pt

{\rm (1)} \ $\mathcal{PGF}=\mathcal{GP}.$

\vskip5pt

{\rm (2)} \ $\mathcal{PGF}^\perp\cap\mathcal{GP}_n=\mathcal P_n$ for any non-negative integer $n$.

\vskip5pt

{\rm (3)} \ $\mathcal{PGF}^\perp\cap\mathcal{GP}_n=\mathcal P_n$ for some non-negative integer $n$.
\end{cor}

\begin{proof}
$(1)\Longrightarrow (2)$: \ By Theorem \ref{mainthm3} one has $\mathcal{PGF}^\perp\cap\mathcal{GP}_n=\mathcal{GP}^\perp\cap\mathcal{GP}_n=\mathcal P_n$, for any non-negative integer $n$.

\vskip5pt

$(2)\Longrightarrow (3)$: \ This is trivial.

\vskip5pt

$(3)\Longrightarrow (1)$: \ Suppose that $\mathcal{PGF}^\perp\cap\mathcal{GP}_n=\mathcal P_n$ for some non-negative $n$. By Corollary \ref{mainthm6}(2) and Theorem \ref{gpctp}(1), there are two cotorsion pairs $(\mathcal{PGF}, \ \mathcal{PGF}^\perp\cap\mathcal{GP}_n)=(\mathcal{PGF}, \ \mathcal P_{n})$ and $(\mathcal{GP},\ \mathcal P_{n})$ in $\mathcal{GP}_n$. Thus $\mathcal{PGF}=\mathcal{GP}$.
\end{proof}

\subsection{Exact model structures on $\mathcal{G P}^{<\infty}$}
It is clear that $\mathcal P_m\cap \mathcal{GP}^{<\infty} = \mathcal P_{m}$.
Also note that $\mathcal{GP}_m\cap \mathcal P^{<\infty} = \mathcal P_{m}$ (cf. \cite[Proposition 2.27]{Ho}).
Similar to Corollary \ref{mainthm6} and Theorem \ref{gpctp} one has

\begin{cor} \label{mainthm7} \ \ Let $R$ be a ring,  and $m$ non-negative integers. Then

\vskip5pt

{\rm (1)} \ The pair \ $(\mathcal{P}_m, \ \mathcal P_m^\perp\cap \mathcal{GP}^{<\infty})$ is a complete and hereditary cotorsion pair in weakly idempotent complete exact category $\mathcal{GP}^{<\infty}$, with
kernel $\mathcal{P}_m\cap\mathcal{P}_m^{\perp}$.

\vskip5pt

{\rm (2)} \ The pair \ $(\mathcal{PGF}_m, \ \mathcal P_m^\perp\cap \mathcal {PGF}^\perp\cap \mathcal{GP}^{<\infty})$ is a complete and hereditary cotorsion pair in  $\mathcal{GP}^{<\infty}$, with
kernel $\mathcal{P}_m\cap\mathcal{P}_m^{\perp}$.

\vskip5pt

{\rm (3)} \ The pair \ $(\mathcal{GP}_m, \ \mathcal P_m^\perp\cap\mathcal{P}^{<\infty})$ is a complete and hereditary cotorsion pair in $\mathcal{G P}^{<\infty}$, with
kernel $\mathcal{P}_m\cap\mathcal{P}_m^{\perp}$.

\vskip 5pt

{\rm (4)}  \ The triple \ $(\mathcal{PGF}_m, \ \mathcal P_m^\perp\cap\mathcal{G P}^{<\infty}, \ \mathcal{PGF}^\perp\cap \mathcal{GP}^{<\infty})$ is a hereditary Hovey triple in $\mathcal{G P}^{<\infty}$.
The corresponding homotopy category is
\ $(\mathcal{PGF}_m\cap \mathcal P_m^\perp) /(\mathcal{P}_m\cap \mathcal P_m^\perp)\cong \mathcal{PGF}/\mathcal{P}.$

\vskip 5pt

{\rm (5)}  \ The triple \ $(\mathcal{G P}_m, \ \mathcal P_m^\perp\cap\mathcal{G P}^{<\infty}, \ \mathcal P^{<\infty})$ is a hereditary Hovey triple in $\mathcal{G P}^{<\infty}$.
The corresponding homotopy category is
\ $(\mathcal{G P}_m\cap \mathcal P_m^\perp) /(\mathcal{P}_m\cap \mathcal P_m^\perp)\cong \mathcal{G P}/\mathcal{P}.$
\end{cor}

\section{\bf Exact model structures on $\mathcal {GF}_n$ and $\mathcal {GF}^{<\infty}$}

\subsection{Application 4: Exact model structures on $\mathcal{GF}_n$}

The following result is due to R. El Maaouy \cite{M}.
It is again a theorem of \v{S}aroch - \v{S}\'{t}ov\'{i}\v{c}ek type; and when $n=0$ it is given in \cite[Theorem 4.11]{SS2}.

\begin{thm} \label{keylemgf}  {\rm(\cite[Theorem 3.4]{M})} \ Let $R$ be a ring, $M$ an $R$-module, and $n$ a non-negative integer.
Then $\mathcal {PGF}^{\perp}\cap\mathcal{GF}_n  = \mathcal{F}_n;$ and the following are equivalent$:$

\vskip5pt

{\rm (1)} \ $M\in \mathcal{GF}_n.$

\vskip5pt

{\rm (2)} \ There is an exact sequence
$0 \longrightarrow K \longrightarrow G \longrightarrow M \longrightarrow 0$ which is again exact after applying  $\operatorname{Hom}_R(-, X)$ for $X\in \mathcal F_n^\perp$, where $G \in \mathcal{PGF}$, and $K$ is flat if $n=0$, and $K \in \mathcal{F}_{n-1}$ if $n\ge 1$.

\vskip5pt

{\rm (3)} \ $\operatorname{Ext}_R^1(M, C)=0$ for every module $C \in \mathcal{F}_n^{\perp}\cap \mathcal {PGF}^{\perp}.$

\vskip5pt

{\rm (4)} \ There is an exact sequence \ $0\longrightarrow M \longrightarrow L\longrightarrow  N\longrightarrow 0$ with $L \in \mathcal{F}_n$ and $N \in \mathcal{PGF}$.
\end{thm}

\vskip5pt

The careful reader will find that the statement (2)
is slightly stronger than  \cite[Theorem 3.4]{M}: here $K \in \mathcal{F}_{n-1}$ if $n\ge 1$, rather than
$K \in \mathcal{F}_{n}$ for all $n\ge 0$ as in \cite[Theorem 3.4]{M}; also, $X\in \mathcal F_n^\perp$ in (2), rather than  $X\in \mathcal F_n\cap \mathcal F_n^\perp$ as in
\cite[Theorem 3.4]{M}. In this way, Theorem \ref{keylemgf} coincides with the corresponding versions for
modules in $\mathcal{PGF}_n$ as in Theorem \ref{mainthm1}, and for modules $\mathcal{GP}_n$ as in Theorem \ref{mainthm3}.
We stress that the proof of \cite[Theorem 3.4]{M} completely holds for Theorem \ref{keylemgf}, thus we will not include a proof.

\vskip5pt

When $n = 0$, the following result is \cite[Corollary 4.12]{SS2}.

\begin{cor} \label{ctpGFm} \ {\rm (\cite[Theorems A, B]{M})} \ Let $R$ be a ring and $n$ a non-negative integer. Then

\vskip5pt

{\rm (1)} \ The pair \ $(\mathcal{GF}_n, \  \mathcal{F}_n^{\perp} \cap \mathcal {PGF}^{\perp})$ is a complete and hereditary cotorsion pair in $R\mbox{-}{\rm Mod}$ with
kernel $\mathcal{F}_n\cap\mathcal{F}_n^{\perp}$. This cotorsion pair is cogenerated by a set.

\vskip 5pt

{\rm (2)}  \ The triple \ $(\mathcal{GF}_n, \ \mathcal{F}_n^{\perp}, \ \mathcal {PGF}^{\perp})$ is a hereditary Hovey triple in $R\mbox{-}{\rm Mod};$
and the corresponding homotopy category is $(\mathcal{GF}_n\cap \mathcal{F}_n^{\perp})/(\mathcal  F_n\cap \mathcal{F}_n^{\perp})\cong \mathcal{PGF}/\mathcal{P}.$
\end{cor}

\begin{lem} \label{ctpMD} \ {\rm (\cite[Theorem 3.4(2)]{MD})} \ Let $R$ be a ring, and $n$ a non-negative integer. Then
$(\mathcal{F}_n, \ \mathcal{F}_n^{\perp})$ is a complete and hereditary cotorsion pair in $R$\mbox{-}{\rm Mod}, and cogenerated by a set.\end{lem}

\vskip5pt

The following result is a special case of D. Bennis, R. El Maaouy, J.R. Garc\'ia Rozas, L. Oyonarte \cite{BMGO}, by taking $W = R$ in Proposition 7.9.

\vskip5pt

\begin{lem} \label{Gflatdim} \ {\rm (\cite[Proposition 7.9]{BMGO})} \ Let $R$ be a ring, and
$0\longrightarrow M_1 \longrightarrow M_2\longrightarrow M_3\longrightarrow 0$ an exact sequence. Then

\vskip5pt

${\rm (i)}$ \ \ ${\rm Gfd}M_1\le {\rm max}\{{\rm Gfd}M_2, {\rm Gfd}M_3-1\};$

 \vskip5pt

${\rm (ii)}$ \  ${\rm Gfd}M_2\le {\rm max}\{{\rm Gfd}M_1, {\rm Gfd}M_3\};$

\vskip5pt

${\rm (iii)}$ \  ${\rm Gfd}M_3\le {\rm max}\{{\rm Gfd}M_1+1, {\rm Gfd}M_2 \}$.
\end{lem}

We specially appreciate the best upper bound given in (i), as for projective dimension, rather than
a correct upper bound: ${\rm Gfd}M_1\le {\rm max}\{{\rm Gfd}M_2, {\rm Gfd}M_3\}.$ This is really needed in the proof of
Theorem \ref{keylemgf}(2).

\vskip10pt

\begin{lem} \label{PGF+ffd =f}  \ Let $R$ be a ring, and $M$ an $R$-module with ${\rm fd} M < \infty$.

\vskip5pt

{\rm (1)} \ {\rm(\cite[Remark 1.5]{E})} \ If $M\in\mathcal {GF}$, then $M$ is flat.

\vskip5pt

{\rm (2)} \ {\rm (\cite{DE})} \ If $M\in\mathcal {PGF}$, then $M$ is projective.

\end{lem}
\begin{proof} \ The assertion (2) is implicitly contained in \cite{DE}. For convenience we include a justification. By \cite[Proposition 2.5{\rm(ii)}]{DE}, ${\rm pd} M = {\rm fd} M < \infty$.
Since $M\in\mathcal {PGF}$, it follows that $M$ is Gorenstein projective (\cite[Theorem 4.4]{SS2}), with finite projective dimension, and hence $M$ is projective (cf. \cite[Proposition 10.2.3]{EJ}).
{\bf An alternative proof:} by Theorem \ref{keylemgf} one has
\vskip5pt
$M\in \mathcal {PGF}\cap \mathcal F_n = \mathcal {PGF}\cap \mathcal{GF}_n\cap \mathcal {PGF}^{\perp}
= (\mathcal {PGF}\cap \mathcal {PGF}^{\perp})\cap \mathcal{GF}_n = \mathcal{P} \cap \mathcal{GF}_n = \mathcal{P}.$\end{proof}

\begin{lem} \label{fnthickingfn} \ Let $R$ be a ring and $n$ a non-negative integer. Then
$\mathcal F_n$ is thick in $\mathcal {GF}_n$.
\end{lem}
\begin{proof} \ If $n = 0$ this in known in \cite[Theorem 4.3]{DE}, by using Lemma \ref{PGF+ffd =f}(1). Assume that $n\ge 1$. The proof is the similar as Lemma \ref{pnthickinpgfn}. Let $0\longrightarrow X\longrightarrow Y\longrightarrow Z\longrightarrow 0$ be an exact sequence in $\mathcal {GF}_n$.
It suffices to prove that if $X\in \mathcal F_n$ and $Y\in \mathcal F_n$ then $Z\in \mathcal F_n$.
By $\mathcal {PGF}^{\perp}\cap\mathcal{GF}_n  = \mathcal{F}_n$ (cf. Theorem \ref{keylemgf})
one has $Z\in \mathcal F_{n+1}\cap \mathcal {GF}_n = \mathcal {PGF}^{\perp}\cap\mathcal{GF}_{n+1} \cap \mathcal {GF}_n
= \mathcal {PGF}^{\perp}\cap \mathcal {GF}_n = \mathcal F_n.$
\end{proof}

\vskip5pt

For any non-negative integer $n$,  we have four chains of complete and hereditary cotorsion pairs $(\mathcal X, \ \mathcal Y)$ in $R$-Mod with
$\mathcal X\subseteq \mathcal B = \mathcal{GF}_n$:
$$(\mathcal P_m, \ \mathcal P_m^\perp), \ \ \ \ (\mathcal F_m, \ \mathcal F_m^\perp), \ \ \ \ (\mathcal {PGF}_m, \ \mathcal P_m^\perp\cap \mathcal {PGF}^\perp),
\ \ \ \ (\mathcal {GF}_m, \ \mathcal F_m^\perp\cap\mathcal {PGF}^{\perp})$$
where $0\le m\le n$ (see Lemma \ref{ctp7.4.6}, Lemma \ref{ctpMD}, Corollary \ref{mainthm2}, and Corollary \ref{ctpGFm}). Applying Theorem \ref{inducedctp1} to $\mathcal B = \mathcal{GF}_n$, one gets
the assertions (1) - (4) in the following corollary.

\begin{cor} \label{mainthm8} \ Let $R$ be a ring,  $m$ and $n$ non-negative integers with $m\le n$. Then

\vskip5pt

{\rm (1)} \ The pair \ $(\mathcal{P}_m, \ \mathcal P_m^\perp\cap \mathcal{GF}_n)$ is a complete and hereditary cotorsion pair in weakly idempotent complete exact category $\mathcal{GF}_n$, with
kernel $\mathcal{P}_m\cap\mathcal{P}_m^{\perp}$.

\vskip5pt

{\rm (2)} \ The pair \ $(\mathcal{F}_m, \ \mathcal F_m^\perp\cap \mathcal{GF}_n)$ is a complete and hereditary cotorsion pair in $\mathcal{GF}_n$, with
kernel $\mathcal{F}_m\cap\mathcal{F}_m^{\perp}$.

\vskip5pt

{\rm (3)} \ The pair \ $(\mathcal{PGF}_m, \ \mathcal{F}_n\cap \mathcal P_m^\perp)$ is a complete and hereditary cotorsion pair in  $\mathcal{GF}_n$, with
kernel $\mathcal{P}_m\cap\mathcal{P}_m^{\perp}$.

\vskip5pt

{\rm (4)} \ The pair \ $(\mathcal{GF}_m, \ \mathcal{F}_n\cap\mathcal F_m^\perp)$ is a complete and hereditary cotorsion pair in $\mathcal{GF}_n$, with
kernel $\mathcal{F}_m\cap\mathcal{F}_m^{\perp}$.

\vskip 5pt

{\rm (5)}  \ The triple \ $(\mathcal{PGF}_m, \ \mathcal P_m^\perp\cap\mathcal{GF}_n, \ \mathcal {F}_{n})$ is a hereditary Hovey triple in $\mathcal{GF}_n$.
The corresponding homotopy category is \ $(\mathcal{PGF}_m\cap \mathcal P_m^\perp) /(\mathcal{P}_m\cap \mathcal P_m^\perp)\cong \mathcal{PGF}/\mathcal P.$

\vskip 5pt

{\rm (6)}  \ The triple \ $(\mathcal{GF}_m, \ \mathcal F_m^\perp\cap\mathcal{GF}_n, \ \mathcal {F}_{n})$ is a hereditary Hovey triple in $\mathcal{GF}_n$.
The corresponding homotopy category is \ $(\mathcal{GF}_m\cap \mathcal F_m^\perp) /(\mathcal{F}_m\cap \mathcal F_m^\perp)\cong \mathcal{PGF}/\mathcal P.$

\end{cor}

\begin{proof} \ Note that $\mathcal{GF}_n$ is closed under extensions (cf. Lemma \ref{Gflatdim}(ii)), and closed under direct summands (cf. \cite[Corollary 7.8(2)]{BMGO}).
Thus $\mathcal{GF}_n$ is a weakly idempotent complete exact category.
Also, $\mathcal{GF}_n$ is closed under the kernels of epimorphisms (cf. Lemma \ref{Gflatdim}(i)): this is a requirement for the applications of Theorem \ref{inducedctp1}.

\vskip 5pt

(1) \ Applying Theorem \ref{inducedctp1} to $\mathcal B = \mathcal{GF}_n$ and $(\mathcal X, \ \mathcal Y) = (\mathcal P_m, \ \mathcal P_m^\perp)$,
one gets (1).

\vskip 5pt

(2) \ Applying Theorem \ref{inducedctp1} to $\mathcal{GF}_n$ and  $(\mathcal X, \ \mathcal Y) =(\mathcal {F}_m, \ \mathcal F_m^\perp)$, one gets (2).

\vskip 5pt

(3) \ Applying Theorem \ref{inducedctp1} to $\mathcal{GF}_n$ and  $(\mathcal X, \ \mathcal Y) =(\mathcal {PGF}_m, \ \mathcal P_m^\perp\cap \mathcal {PGF}^\perp)$, one gets
the complete and hereditary cotorsion pair $(\mathcal {PGF}_m, \ \mathcal P_m^\perp\cap \mathcal {PGF}^\perp\cap \mathcal{GF}_n)$ in $\mathcal{GF}_n$.
Since $\mathcal{PGF}^\perp \cap\mathcal{GF}_n = \mathcal F_n$ (cf. Theorem \ref{keylemgf}), this cotorsion pair is just $(\mathcal{PGF}_m, \ \mathcal{F}_n\cap \mathcal P_m^\perp)$.

\vskip 5pt

Since $\mathcal{PGF}_m\cap\mathcal{F}_n = \mathcal{PGF}_m\cap \mathcal{PGF}^\perp \cap\mathcal{GF}_n = \mathcal{PGF}_m\cap \mathcal{PGF}^\perp  = \mathcal P_m$
(cf. Theorem \ref{mainthm1}), the kernel of this cotorsion pair is $\mathcal{P}_m\cap\mathcal P_m^\perp$.

\vskip 5pt

(4) \ Applying Theorem \ref{inducedctp1} to $\mathcal{GF}_n$ and $(\mathcal X, \ \mathcal Y) = (\mathcal {GF}_m, \ \mathcal F_m^\perp\cap\mathcal {PGF}^{\perp})$,
one gets the complete and hereditary cotorsion pair $(\mathcal {GF}_m, \ \mathcal F_m^\perp\cap \mathcal {PGF}^{\perp}\cap \mathcal{GF}_n)$ in $\mathcal{GF}_n$. By
$\mathcal{PGF}^\perp \cap\mathcal{GF}_n = \mathcal F_n$ (cf. Theorem \ref{keylemgf}), this cotorsion pair is just
$(\mathcal {GF}_m, \ \mathcal F_n\cap \mathcal{F}^\perp_m)$.

\vskip 5pt

Since $\mathcal {GF}_m\cap \mathcal{F}_n = \mathcal {GF}_m\cap \mathcal {PGF}^{\perp}\cap \mathcal{GF}_n = \mathcal {GF}_m\cap \mathcal {PGF}^{\perp} = \mathcal{F}_m$,
the kernel of this cotorsion pair is
$\mathcal{F}_m\cap\mathcal F_m^\perp$.

\vskip 5pt

(5) \ By Lemma \ref{fnthickingfn}, $\mathcal {F}_{n}$ is thick in $\mathcal{GF}_n$.

\vskip 5pt

Since $\mathcal {PGF}_m\cap \mathcal{F}_n = \mathcal P_m$, it follows from (1) and (3)
that $(\mathcal{PGF}_m, \ \mathcal P_m^\perp\cap\mathcal{GF}_n, \ \mathcal {F}_{n})$ is a hereditary Hovey triple in $\mathcal{GF}_n$.
Thus, by Theorem \ref{hoherhov}, $\mathcal{PGF}_m\cap\mathcal P_m^\perp\cap\mathcal{GF}_n = \mathcal{PGF}_m\cap\mathcal P_m^\perp$ is a Frobenius category,
$\mathcal{PGF}_m\cap\mathcal P_m^\perp\cap\mathcal{GF}_n\cap\mathcal {F}_{n} = \mathcal {P}_{m}\cap\mathcal P_m^\perp$ is the class of its projective-injective objects,
and   the homotopy category of the corresponding exact model structure on $\mathcal{GF}_n$ is \ $(\mathcal{PGF}_m\cap \mathcal P_m^\perp) /(\mathcal{P}_m\cap \mathcal P_m^\perp).$  It is triangle equivalent to $\mathcal{PGF}/\mathcal P,$ by using Corollary \ref{wdeterminetriho} and
hereditary Hovey triple $(\mathcal{PGF}, \ \mathcal{GF}_n, \ \mathcal {F}_n)$  in $\mathcal{GF}_n$.

\vskip 5pt

(6) \ By Lemma \ref{fnthickingfn}, $\mathcal {F}_{n}$ is thick in $\mathcal{GF}_n$.

\vskip 5pt

Since $\mathcal {GF}_m\cap \mathcal{F}_n = \mathcal {F}_m$, it follows from (2) and (4)
that \ $(\mathcal{GF}_m, \ \mathcal F_m^\perp\cap\mathcal{GF}_n, \ \mathcal {F}_{n})$ is a hereditary Hovey triple in $\mathcal{GF}_n$.
Thus, by Theorem \ref{hoherhov}, $\mathcal{GF}_m\cap\mathcal F_m^\perp$ is a Frobenius category,
$\mathcal {F}_{m}\cap\mathcal F_m^\perp$ is the class of its projective-injective objects,
and   the homotopy category of the corresponding exact model structure on $\mathcal{GF}_n$ is \ $(\mathcal{GF}_m\cap \mathcal F_m^\perp) /(\mathcal{F}_m\cap \mathcal F_m^\perp).$  By (5), $(\mathcal{PGF}, \ \mathcal{GF}_n, \ \mathcal {F}_n)$ is also a hereditary Hovey triple in $\mathcal{GF}_n$,
with homotopy category  $\mathcal{PGF}/\mathcal P.$ By Corollary \ref{wdeterminetriho}, there is a triangle equivalence
$(\mathcal{GF}_m\cap \mathcal F_m^\perp) /(\mathcal{F}_m\cap \mathcal F_m^\perp)\cong \mathcal{PGF}/\mathcal P$.
\end{proof}

When $n = m =0$, Corollary \ref{mainthm8}(3) and (5) has been obtained by G. Dalezios and I. Emmanouil \cite[Theorem 4.3]{DE}.

\vskip5pt

\subsection{Application 5: Exact model structures on $\mathcal{GF}^{<\infty}$}
\begin{fact} \ Let $R$ be a ring and $m$ a non-negative integer. Then

\vskip5pt

{\rm (1)} \ $\mathcal{PGF}^{\perp}\cap \mathcal {GF}^{<\infty}  = \mathcal F^{<\infty}.$

\vskip5pt

{\rm (2)} \ $\mathcal{PGF}_m\cap \mathcal {F}^{<\infty}  = \mathcal P_m.$

\vskip5pt

{\rm (3)} \ \ $\mathcal{GF}_m\cap \mathcal {F}^{<\infty}  = \mathcal F_m.$

\end{fact}
\begin{proof} \ One has $\mathcal{PGF}^\perp \cap\mathcal{PGF}_t = \mathcal P_t$ (cf. Theorem \ref{mainthm1}) and $\mathcal{PGF}^\perp \cap\mathcal{GF}_t = \mathcal F_t$ (cf. Theorem \ref{keylemgf}), for any non-negative integer $t$.

\vskip5pt
{\rm (1)} \ One has $\mathcal{PGF}^{\perp}\cap \mathcal {GF}^{<\infty}  =  \bigcup\limits_{t\ge 0}(\mathcal{PGF}^{\perp}\cap \mathcal {GF}_t) =  \bigcup\limits_{t\ge 0}\mathcal {F}_t
= \mathcal F^{<\infty}.$

\vskip5pt

{\rm (2)} \ One has $\mathcal{PGF}_m\cap \mathcal {F}^{<\infty} =  \bigcup\limits_{t\ge 0}(\mathcal{PGF}_m\cap \mathcal {F}_t)
 =  \bigcup\limits_{t\ge 0}(\mathcal{PGF}_m\cap \mathcal{PGF}^\perp \cap\mathcal{GF}_t) = \bigcup\limits_{t\ge 0}(\mathcal{P}_m\cap\mathcal{GF}_t) = \mathcal{P}_m\cap\mathcal {GF}^{<\infty}  = \mathcal P_m.$

\vskip5pt

{\rm (3)} \ One has $\mathcal{GF}_m\cap \mathcal {F}^{<\infty} =  \bigcup\limits_{t\ge 0}(\mathcal{GF}_m\cap \mathcal {F}_t)
 =  \bigcup\limits_{t\ge 0}(\mathcal{GF}_m\cap \mathcal{PGF}^\perp \cap\mathcal{GF}_t) = \bigcup\limits_{t\ge 0}(\mathcal{F}_m\cap\mathcal{GF}_t) = \mathcal{F}_m\cap\mathcal {GF}^{<\infty}  = \mathcal F_m.$
\end{proof}

By the similar argument as in Corollary \ref{mainthm8} one has

\begin{cor} \label{mainthm9} \ Let $R$ be a ring and $m$ a non-negative integers. Then

\vskip5pt

{\rm (1)} \ The pair \ $(\mathcal{P}_m, \ \mathcal P_m^\perp\cap \mathcal{GF}^{<\infty})$ is a complete and hereditary cotorsion pair in weakly idempotent complete exact category $\mathcal{GF}^{<\infty}$, with
kernel $\mathcal{P}_m\cap\mathcal{P}_m^{\perp}$.

\vskip5pt

{\rm (2)} \ The pair \ $(\mathcal{F}_m, \ \mathcal F_m^\perp\cap \mathcal{GF}^{<\infty})$ is a complete and hereditary cotorsion pair in $\mathcal{GF}^{<\infty}$, with
kernel $\mathcal{F}_m\cap\mathcal{F}_m^{\perp}$.

\vskip5pt

{\rm (3)} \ The pair \ $(\mathcal{PGF}_m, \ \mathcal{F}^{<\infty}\cap \mathcal P_m^\perp)$ is a complete and hereditary cotorsion pair in  $\mathcal{GF}^{<\infty}$, with
kernel $\mathcal{P}_m\cap\mathcal{P}_m^{\perp}$.

\vskip5pt

{\rm (4)} \ The pair \ $(\mathcal{GF}_m, \ \mathcal{F}^{<\infty}\cap\mathcal F_m^\perp)$ is a complete and hereditary cotorsion pair in $\mathcal{GF}^{<\infty}$, with
kernel $\mathcal{F}_m\cap\mathcal{F}_m^{\perp}$.

\vskip 5pt

{\rm (5)}  \ The triple \ $(\mathcal{PGF}_m, \ \mathcal P_m^\perp\cap\mathcal{GF}^{<\infty}, \ \mathcal {F}^{<\infty})$ is a hereditary Hovey triple in $\mathcal{GF}^{<\infty}$.
The corresponding homotopy category is \ $(\mathcal{PGF}_m\cap \mathcal P_m^\perp) /(\mathcal{P}_m\cap \mathcal P_m^\perp)\cong \mathcal{PGF}/\mathcal P.$

\vskip 5pt

{\rm (6)}  \ The triple \ $(\mathcal{GF}_m, \ \mathcal F_m^\perp\cap\mathcal{GF}^{<\infty}, \ \mathcal {F}^{<\infty})$ is a hereditary Hovey triple in $\mathcal{GF}^{<\infty}$.
The corresponding homotopy category is \ $(\mathcal{GF}_m\cap \mathcal F_m^\perp) /(\mathcal{F}_m\cap \mathcal F_m^\perp)\cong \mathcal{PGF}/\mathcal P.$
\end{cor}

\vskip20pt

{\bf Acknowledgement}: We thank Ioannis Emmanouil for his useful comments.


\vskip10pt

Nan Gao, \ \ \ Department of Mathematics, \ \ Shanghai University,  \ Shanghai 200444, \ China

\vskip5pt

Xue-Song Lu, \ \ \ \ Pu Zhang, \ \ School of Mathematical Sciences, \ \ Shanghai Jiao Tong University,  \ Shanghai 200240, \ China

\end{document}